\documentclass{amsart}
\usepackage{amsmath,amsthm}
\usepackage{mathrsfs}
\usepackage{amsfonts,amssymb}
\usepackage{accents}
\usepackage{enumerate}
\usepackage{accents,color}
\usepackage{graphicx}

\newcommand{\lvt}{\left|\kern-1.35pt\left|\kern-1.3pt\left|}
\newcommand{\rvt}{\right|\kern-1.3pt\right|\kern-1.35pt\right|}

\newtheorem{thm}{Theorem}[section]
\newtheorem{cor}[thm]{Corollary}
\newtheorem{lem}[thm]{Lemma}
\newtheorem{prop}[thm]{Proposition}

\newtheorem{ass}[thm]{Assumption}

\theoremstyle{remark}
\newtheorem{rem}{Remark}[section]

 \def\la{{\langle}}
 \def\ra{{\rangle}}
 \def\ve{{\varepsilon}}

 \def\d{\mathrm{d}}
 \def\e{\mathrm{e}}
 \def\i{\mathrm{i}}
 
 \def\sph{{\mathbb{S}^{d-1}}}

 \def\a{{\alpha}}
 \def\b{{\beta}}
 \def\g{{\gamma}}
 \def\k{{\kappa}}
 \def\bka{\boldsymbol{\kappa}}
 \def\t{{\theta}}
 \def\l{{\lambda}}
 
 \def\s{\sigma}
 \def\la{{\langle}}
 \def\ra{{\rangle}}
 \def\ve{{\boldsymbol{\varepsilon}}}

 \def\jb{{\mathbf j}}
 \def\kb{{\mathbf k}}
 
 \def\mb{{\mathbf m}}

 \def\CH{{\mathcal H}}

 \def\CV{{\mathcal V}}

 \def\BB{{\mathbb B}}
 \def\CC{{\mathbb C}}
 
 \def\NN{{\mathbb N}}

 \def\RR{{\mathbb R}}
 \def\SS{{\mathbb S}}
 
 \def\TT{{\mathbb T}}
 
 \def\ZZ{{\mathbb Z}}

\def\lla{\langle{\kern-2.5pt}\langle}      
\def\rra{\rangle{\kern-2.5pt}\rangle}      
\def\balpha{\boldsymbol{\alpha}}
\def\bbeta{\boldsymbol{\beta}}

\newcommand{\wh}{\widehat}

\def\one{{\mathbf{1}}}
\def\f{\frac}

\graphicspath{{./}}
\begin{document}
 
\def\x{\boldsymbol{x}}
\def\y{\boldsymbol{y}}
\def\m{\boldsymbol{m}}
\def\a{\boldsymbol{a}}
\def\b{\boldsymbol{b}}
\def\e{\boldsymbol{e}}
\def\u{\boldsymbol{u}}

\def\A{\mathbf{A}}
\def\B{\mathbf{B}}
\def\bX{\mathbf{X}}
\def\bq{\mathbf{q}}
\def\bp{\mathbf{p}}
\def\bxi{\boldsymbol{\xi}}
\def\bzeta{\boldsymbol{\zeta}}

\def\v{\boldsymbol{v}}
\def\z{\boldsymbol{z}}
\def\p{\boldsymbol{p}}
\def\bS{\mathbf{S}}
\def\K{\mathbf{K}}
\def\P{\mathbf{P}}
\def\M{\mathbf{M}}
\def\Q{\mathbf{Q}}
\def\N{\mathbb{N}}
\def\R{\mathbb{R}}
\def\RR{\mathbb{R}}
\def\SS{\mathbb{S}}
\def\Z{\mathbb{Z}}
\def\M{\mathbf{M}}
\def\om{\mathbf{\Omega}}
\def\bphi{\boldsymbol{\phi}}
\def\bxi{\boldsymbol{\xi}}
\def\bmu{\boldsymbol{\mu}}
\def\btheta{\boldsymbol{\theta}}
\def\balpha{\boldsymbol{\alpha}}
\def\bbeta{\boldsymbol{\beta}}
\def\bgamma{\boldsymbol{\gamma}}
\def\taue{\tau_{\varepsilon}}
\def\hh{\overline{H}}
\def\h{\overline{h}}
\def\g{\overline{g}}
\def\tauer{\tau_{\varepsilon/r}}
\def\1{\mathbf{1}}
\def\bk{\mathbf{k}}
\def\t{\theta}

\title[A generalized Pell's equation]
{A Generalized Pell's equation for a class of multivariate orthogonal polynomials}
\author{Jean B. Lasserre}
\address{LAAS-CNRS and Toulouse School of Economics (TSE)\\ LAAS, 7 avenue du Colonel Roche\\ 31077 Toulouse C\'edex 4, France} \email{lasserre@laas.fr}
\author{Yuan Xu}
\address{Department of Mathematics\\ University of Oregon\\ Eugene, OR 97403-1222, U.S.A.} \email{yuan@uoregon.edu}

\date{}

\thanks{The first author's work is supported by the AI Interdisciplinary Institute ANITI funding through the French program ``Investing for the Future PI3A" under grant agreement number ANR-19-PI3A-0004
and the research is also part of the program DesCartes supported by the National Research Foundation,  Prime Minister's Office, Singapore under its Campus for Research Excellence and Technological Enterprise (CREATE) program. The second author's work is partially supported by Simons Foundation Grant \# 849676.}

\begin{abstract}
 We extend the polynomial Pell's equation satisfied by univariate Chebyshev polynomials on $[-1,1]$ from one
 variable to several variables, using orthogonal polynomials on regular domains that include cubes, balls, and
 simplexes of arbitrary dimension. Moreover, we show that such an equation is 
 strongly connected (i) to  a certificate of positivity (from real algebraic geometry) 
 on the domain, as well as (ii) to the  Christoffel functions of the equilibrium measure on the domain. In addition, the solution to Pell's equation 
 reflects an extremal property of orthonormal polynomials associated with an entropy-like criterion. 
 \end{abstract}
\maketitle

\section{Introduction}
\setcounter{equation}{0}

The starting point of our investigation is polynomial Pell's equation 
\begin{equation}\label{intro-0}
T_n(x)^2+(1-x^2)\,U_{n-1}(x)^2\,=\,1\,,\quad\forall x\in \R,\quad n\in\N\,,
\end{equation}
where $T_n \in \Z[x]$ and $U_n \in \Z[x]$ are the Chebyshev polynomial of the first and second kinds, respectively. 
We aim at extending this identity to multivariate polynomials in the form of 
\begin{equation}\label{intro-0B}
 P^\mu_n(\x)+\sum_{i=1}^{r}\phi_i(\x)\,P^{\phi_i\cdot\mu}_{n,i}(\x)\,=\,1, \quad\forall \x \in \R^d, \quad n \in \N,
\end{equation}
where $\phi_i$'s are some products of generators of a domain $\Omega\subset\R^d$, $P^\mu_n$ (resp. $P^{\phi_i\cdot\mu}_{n,i}$) is 
the sum of the squares of all orthonormal polynomials of degree $n$ (resp. $n-\mathrm{deg}(\phi_i)$) with respect to the 
equilibrium measure $\mu$ on $\Omega$ (resp. the measure $\phi_i\cdot\mu =\phi_id\mu$ with the same support $\Omega$).

A theoretical study of this problem has been carried out recently in \cite{cras}, where it is pointed out that the problem
is connected to several other topics, which we describe below as they motivate our study. We first define the Christoffel
function. Let $\mu$ be a measure supported on $\Omega \in \RR^d$. For $n \in \NN$, let $\R[\x]_n$ be the 
space of polynomials of degree $n$ in $\x \in \RR^d$. The Christoffel function associated with $\mu$, denoted by 
$\Lambda^\mu$, is defined by  
\begin{equation} \label{eq:variational}
 \Lambda^\mu_n(\bxi) : = \inf_{p\in\R[\x]_n, p(\bxi)=1} \int p(\y)^2 \,d\mu(\y),\qquad\forall\bxi\in\R^d.
\end{equation}

Let $x\mapsto g(x):=1-x^2$, $s(n):= \binom{1+n}{n}$, and let $\mu$ be the Chebyshev (equilibrium) measure 
$dx/\pi\sqrt{1-x^2}$ of $[-1,1]$, whereas $g\cdot\mu:=\sqrt{1-x^2}dx/\pi$. After scaling, a summation of \eqref{intro-0} 
leads to the identity 
\begin{equation} \label{intro-1}
\Lambda^\mu_n(x)^{-1}+(1-x^2)\,\Lambda^{g\cdot\mu}_{n-1}(x)^{-1}\,=\,s(n)+s(n-1),
\end{equation}
where $\Lambda^\mu_n$ (resp. $\Lambda^{g\cdot\mu}_n$) is the Christoffel function of degree $n$ associated with  
$\mu$ (resp. $g\cdot\mu$). The Christoffel function can also be formulated in terms of the moments of $\mu$; 
see the next section. For the Chebyshev measure, the formulation is of the form 
\begin{align*}
\Lambda^\mu_n(x)^{-1}\,& =\v_n(x)^T\M_n(\mu)^{-1}\v_n(x)\,,\quad \forall x\in\R,\\
\Lambda^{g\cdot\mu}_{n-1}(x)^{-1} \ &=\v_{n-1}(x)^T\M_{n-1}(g\cdot\mu)^{-1}\v_{n-1}(x)\,,\quad\forall x\in\R,
\end{align*}
where $\v_n(x)=(1,x,x^2,\ldots,x^n)$, $\M_n(\mu)$ is the matrix of moments up to degree $2n$ of $\mu$, 
and  $\M_{n-1}(g\cdot\mu)$ is the matrix of moments up to degree $2n-2$ of $g\cdot\mu$.

While polynomial Pell's equation \eqref{intro-0} originates from Pell's equation in algebraic number theory, it also can be 
regarded from a different angle. In \cite{cras} it was observed that \eqref{intro-0} (resp. \eqref{intro-1}) is the 
\emph{Markov-Luk\'acs certificate} (resp. \emph{Putinar certificate})
that the constant polynomial $x\mapsto p(x)=1$ for all $x$, is positive on $[-1,1]$. Indeed while the former states 
that a polynomial $p$ of degree $2n$, nonnegative on $[-1,1]$, can be written in the form
\[p(x)\,=\,q_0(x)^2+(1-x^2)\,q_1(x)^2\,,\quad\forall x\in\R\,,\]
for two single squares of polynomials $q_0$ of degree $n$ and $q_1$ of degree $n-1$, the latter states
that 
\begin{equation}
    \label{intro-univ-put}
p(x)\,=\,\sigma_0(x)+(1-x^2)\,\sigma_1(x)\,,\quad\forall x\in\R\,,\end{equation}
for two polynomials $\sigma_0$ of degree $2n$ and $\sigma_1$ of degree $2n-2$ that are \emph{sum-of-squares}. 
The latter, importantly, has a multivariate generalization \cite{putinar} to compact basic semi-algebraic sets of $\R^d$ whereas the 
former is specific to the univariate case. Furthermore, the reciprocals of the Christoffel functions $\Lambda^\mu_n$ 
and $\Lambda^{g\cdot\mu}_n$ (both polynomial sum-of-squares) satisfy an extremal property: Namely the couple $(\M_n(\mu)^{-1},\M_{n-1}(g\cdot\mu)^{-1})$ 
of their respective Gram matrices is the unique optimal solution for the optimization problem
\begin{equation}
    \label{intro-logdet}
\begin{array}{rl}  {\displaystyle\max_{\A,\B\succ0}}& \{\log\mathrm{det}(\A)+\log\mathrm{det}(\B): \\
\mbox{s.t.}& 
s(n)+s(n-1)=\,\underbrace{\v_n(x)^T\A\v_n(x)}_{\sigma_0(x)}
       +\underbrace{\v_{n-1}(x)^T\B\v_{n-1}(x)}_{\sigma_1(x)}\,(1-x^2)\\
       &\forall x\in\R\,\}\,,
\end{array}
\end{equation}
where the ``$\sup$" is over all possible positive semidefinite Gram matrices $\A$ and $\B$
of the sum-of-squares polynomials $\sigma_0$ and $\sigma_1$ respectively.
Interestingly, this is a direct consequence and interpretation of a result by  Nesterov  \cite{nesterov} on a one-to-one 
correspondence between the respective interiors of the convex cone 
of polynomials of the form \eqref{intro-univ-put} and its dual; see \cite{cras,nesterov}.

\subsection*{Contribution}
In \cite{cras}, the identity \eqref{intro-0} is extended to \eqref{intro-0B} in several bivariate cases, namely for set 
$\Omega\subset\R^2$ being the Euclidean ball, the triangle, and the cube, but only for $n=1,2,3$. 
In the present paper:

(i) We first show that these extensions hold for all $n \ge 1$ and, furthermore, for all dimensions, not just for $d =2$. 
Together, our results consist of a significant extension of \eqref{intro-0} to the multivariate setting and provide the first 
set of examples for \eqref{intro-0B}. In particular, it shows that \eqref{intro-1}, as a partition of unity for $[-1,1]$, has 
a natural multivariate analog for the regular domains in our list.

The proof of our extensions comes from a special property of multivariate orthogonal polynomials. More specifically, 
the property is an analog of the {\it addition formula} for spherical harmonics, which are homogeneous harmonic 
polynomials restricted on the unit sphere $\SS^d$ of $\RR^{d+1}$. Indeed, let $\CH_n^{d+1}$ be the space of spherical 
harmonics of degree $n$ and let $\{Y_\ell^n: 1 \le \ell \le \dim \CH_n^{d+1}\}$ be an orthonormal basis of $\CH_n^d$; then 
the addition formula states that 
\begin{equation} \label{intro:addtion}
   \sum_{\ell = 1}^{\dim \CH_n^{d+1}} Y_\ell^n(\bxi) Y_\ell^n (\boldsymbol{\eta}) = \frac{2 n+d-1}{d-1} C_n^{\f{d-1}{2}} (\la \bxi,\boldsymbol{\eta}\ra), 
   \quad \bxi, \boldsymbol{\eta} \in \SS^d, 
\end{equation}
where $C_n^\l$ denotes the usual Gegenbauer polynomial of degree $n$ and $\la \cdot,\cdot\ra$ is the ordinary inner
product in $\RR^d$. It shows, in particular, that 
$$
    \sum_{\ell = 1}^{\dim \CH_n^{d+1}} Y_\ell^n(\bxi)^2 = \frac{2 n+d-1}{d-1} C_n^{\f{d-1}{2}} (1),\quad\forall\bxi\in\SS^d\,,
$$
which is a sum of square formula or an analog of \eqref{intro-0B}. For the Chebyshev weight function on each of 
the regular domains in our list, we have an analog of the addition formula, which however is often equal to a sum on 
its right-hand side, instead of one term as in \eqref{intro:addtion}. Nevertheless, in each case, we shall show that 
an appropriate combination with associated orthogonal polynomials for a family of weight functions, involving 
generators of the set $\Omega$, leads to an identity that can be used to obtain \eqref{intro-0B}. The identity is of 
interest in its own as it links all orthonormal polynomials of degree $n$ for $\mu$ with all orthonormal polynomials
of degree $n-\mathrm{deg}(\phi_i)$ for $\phi_i\cdot\mu$, $i=1,\ldots,r$, where $(\phi_i)_{i=1}^r$ is a set of generators
of the domain $\Omega$. 

Moreover, by summing over $n=0,1,\ldots,t$, we obtain an identity satisfied by reciprocals 
of the Christoffel functions associated with measures $\mu$ and $\phi\cdot\mu$. As an analogy with \eqref{intro-1}, 
it is fair to state that these Christoffel functions are solutions of a generalized polynomial Pell's equation for every fixed 
degree $n\in\N$. 

(ii) In addition, the inverses of the associated moment matrices for these measures satisfy an extremal 
property since they also form the unique optimal solution of a multivariate version of the optimization problem 
\eqref{intro-logdet}.

(iii) Next, in the general case of a compact set $\Omega$ with nonempty interior,
and whose set of generators is given and fixed, we show that to its associated equilibrium measure
$\mu$ is associated a sequence of polynomials $(p_n)_{n\in\N}$ of increasing degree, positive on the domain. We show that the sequence $(p_n)$ is related to the constant polynomial $\1$:

- in a weak form when $p_n$ is seen as a density with respect to $\mu$: the resulting sequence of probability 
measures $(p_n\mu)_{n\in\N}$ converges weakly to $\mu$.

- in a stronger form under additional assumptions: the sequence $(p_n)_{n\in\N}$ converges to $\1$, uniformly on 
compact subsets of $\Omega$.\\

Altogether, our generalization of \eqref{intro-1} to certain multivariate settings reveals unexpected links with seemingly 
disconnected fields (orthogonal polynomials, equilibrium measures, certificates of positivity, a certain conic duality 
in optimization), which we hope will be of interest to researchers working in those fields.  

The paper is organized as follows. After a brief section on notations and preliminaries, we discuss the background
and motivation in the third section, then describe our results in relatively simple terms of the Christoffel functions of 
appropriate measures on the family of semi-algebraic sets $\Omega \subset\R^d$ that we consider in the fourth
section. The proof will be given in the fifth section, which contains the full strength of the identity derived from the 
addition formula. The nature of the proof means it is more involved and technique, for which we will introduce
additional notations and restate our results for easier access and clarity of exposition. 

\section{Notation and preliminaries}
\setcounter{equation}{0}
\label{sec:notation}
\subsection*{Notation and definitions}

Let $\R[\x]$ denote the ring of real polynomials in the variables $\x=(x_1,\ldots,x_d)$ and $\R[\x]_n\subset\R[\x]$ be 
the subset of polynomials of total degree at most $n$. Denote by $\1\in\R[\x]$ the polynomial equal to $1$ for all $\x$.
Let $\N^d_n:=\{\balpha\in\N^d:\vert\balpha\vert\leq n\}$, where 
$\vert\balpha\vert=\sum_i\alpha_i$, which has the cardinal $s(n):=\binom{d+n}{d}$. Let 
$\v_n(\x)=(\x^{\balpha})_{\balpha\in\N^d_n}$ be the vector of monomials up to degree $n$, where the monomials are 
listed in, say, the graded lexicographical order, and let $\Sigma[\x]\subset\R[\x]$ (resp. $\Sigma[\x]_n\subset\R[\x]_{2n}$) 
be the convex cone of polynomials (resp. of polynomials of total degree at most $2n$) that are sum-of-squares (SOS in short). 
For every $p\in\R[\x]_n$ write
\[
  \x\mapsto p(\x)\,=\,\langle\p,\v_n(\x)\rangle, \quad \forall \x\in\R^d,
\]
where $\p\in\R^{s(n)}$ is the vector of coefficients of $p$ in the monomial basis $(\x^{\balpha})_{\balpha\in\N^d}$.
For a real symmetric matrix $\A=\A^T$ the notation $\A\succeq0$ (resp. $\A\succ0$) stands for $\A$ is positive semidefinite 
(p.s.d.) (resp. positive definite (p.d.)).

The support of a Borel measure $\mu$ on $\R^d$ is the smallest closed set $A$ such that $\mu(\R^d\setminus A)=0$, 
and such a  set $A$ is unique. With $\Omega \subset\R^d$ compact, denote by $\mathscr{C}(\Omega)$ 
the Banach space of real continuous functions on $\Omega$ equipped with the sup-norm. Its topological dual
$\mathscr{C}(\Omega)^*$ is the Banach space $\mathscr{M}(\Omega)$ of finite signed Borel measures on $\Omega$, 
equipped with the total-variation norm.

\subsection*{Moment matrix}
Associated with $n\in\N$ and a real vector $\bphi=(\phi_{\balpha})_{\balpha\in\N^d_{2n}}$, are  the (Riesz) linear functional $\phi\in\R[\x]_{2n}^*$ defined by:
\[p\mapsto \phi(p)\,:=\,\langle \p,\bphi\rangle\,=\,\sum_{\balpha\in\N^d_{2n}}p_{\balpha}\,\phi_{\balpha}\,,\quad\forall p\in\R[\x]_{2n}\,,\]
and the real ``moment" matrix $\M_n(\bphi)$ (or $\M_n(\phi)$)
with rows and columns indexed by $\balpha\in\N^d_n$ and with entries
\[\M_n(\bphi)(\balpha,\bbeta)\,:=\,\phi(\x^{\balpha+\bbeta})\,=\,\phi_{\balpha+\bbeta}\,\quad\forall\balpha,\bbeta\in\N^d_n\,.\]
If $\mu$ is a finite Borel measure on $\R^d$ with all moments $\bmu=(\mu_{\balpha})_{\balpha\in\N^d}$
assumed to be finite, then
\[
\M_n(\mu)(\balpha,\bbeta) :=\,\int \x^{\balpha+\bbeta}\,d\mu = \mu_{\balpha+\bbeta},\quad\balpha,\bbeta\in\N^d_n\,,
\]
and obviously, $\M_n(\mu)\succeq0$ for all $n$ since 
\[
\langle \p,\M_n(\mu)\,\p\rangle =\int p^2\,d\mu\,\geq\,0\,,\quad\forall p\in\R[\x]_n.
\]
On the other hand, given a real vector
$\bphi=(\phi_{\balpha})_{\balpha\in\N^d_{2n}}$, $\M_n(\bphi)\succeq0$ is only a necessary condition for the associated linear functional $\phi\in\R[\x]_{2n}^*$ to have a representing measure, i.e.,
\[\phi_{\balpha}\,=\,\phi(\x^{\balpha})\,=\,
\int \x^{\balpha}\,d\varphi\,,\quad\forall\balpha\in\N^d_{2n}\,,\]
for some Borel measure $\varphi$ on $\R^d$.

\subsection*{Localizing matrix}

Given a real sequence $\bphi\in\N^d_{2n}$ and a polynomial $g\in\R[\x]$, one may define the new real sequence $g\cdot\bphi\in\R^{s(2n)-\mathrm{deg}(g)}$ (and associated Riesz linear functional $g\cdot\phi\in\R[\x]_{2n-\mathrm{deg}(g)}^*$) by:
\[
(g\cdot\bphi)_{\balpha}\,=\,g\cdot\phi(\x^{\balpha})\,:=\,\phi(g\,\x^{\balpha})\,=\,
\sum_{\bbeta\in\N^d_n}g_{\bbeta}\,\phi_{\balpha+\bbeta}\,,\quad\forall\balpha\in\N^d_{2n-\mathrm{deg}(g)}.
\]
Then, setting $t_g:=\lceil\mathrm{deg}(g)/2\rceil\,(=\lceil t/2\rceil$), the \emph{localizing} matrix $\M_{n-t_g}(g\,\bphi)$ $(n\geq t_g$) 
associated with $g$ and $\bphi$ is just the moment matrix $\M_{n-t_g}(g\cdot\bphi)$ associated with
the (pseudo) moment vector $g\cdot\bphi\in\R^{s(2n)-\mathrm{deg}(g)}$.

\subsection*{Orthogonal polynomials and their kernels.}
Here we assume $\M_n(\mu)\succ0$ for all $n\in\N$, and therefore the inverse $\M_n(\mu)^{-1}$ is well-defined 
for all $n\in\N$. In particular, this is true in our case of interest, i.e.,  when the support $\Omega \subset\R^d$ of $\mu$ is 
compact with a nonempty interior and $\mu$ has a density with respect to the Lebesgue measure on $\Omega$. Under
the assumption, a sequence of orthogonal polynomials $(P_{\balpha})_{\balpha\in\N^d}\subset\R[\x]$ exists in $L^2(\mu)$. 

Let $(P_{\balpha})_{\balpha\in\N^d}\subset\R[\x]$ be a family of polynomials, with the degree of $P_{\balpha}$ being $|\balpha|$,
that are orthonormal with respect to $\mu$; that is, 
$$
  \int_{\RR^d} P_{\balpha}(\x) P_{\bbeta}(\x) \d \mu(\x) = \delta_{\balpha,\bbeta}, \qquad \balpha, \bbeta \in \N^d.
$$
Let $\CV_m^d$ be the space of orthogonal polynomials of degree exactly $m$. Let $\{P_{\balpha}: |\a| = m\}$ be
an orthonormal basis of $\CV_m^d$. Then the reproducing kernel of the space $\CV_m^d$ in $L^2(\mu)$ is defined by 
\begin{equation}\label{eq:P-kernel}
(\x,\y)\mapsto P^\mu_{m}(\x,\y) := \sum_{\vert\balpha\vert=m}P_{\balpha}(\x)\,P_{\balpha}(\y),\quad\forall\x,\y, \quad m\in\N,
\end{equation}
which is independent of the choice of bases. Summing over $0 \le m \le n$ gives the reproducing kernel of $\R[\x]_n$ in 
the space $L^2(\mu)$, 
\begin{equation} \label{eq:K-kernel}
K^\mu_n(\x,\y) =\sum_{j=0}^n  P^\mu_j(\x,\y) =  \sum_{\vert\balpha\vert \le n}P_{\balpha}(\x)\,P_{\balpha}(\y), 
\end{equation}
which is sometimes called the Christoffel-Darboux kernel. 
     
\subsection*{Christoffel function.} 
The Christoffel function $\Lambda^\mu_n:\R\to\R_+$ of degree $n$, associated with $\mu$, is defined by
\eqref{eq:variational} and it can also be defined in terms of the moment matrix by 
\[
\x\mapsto \Lambda^{\mu}_n(\x)^{-1} = \v_n(\x)^T\M_n(\mu)^{-1}\v_n(\x)\,,\quad\forall \x\in\R^d.
\]
Alternatively, in terms of the kernels $K_n^\nu(\cdot,\cdot)$, the Christoffel function satisfies 
\begin{equation} \label{ortho-poly-0}
\Lambda^{\mu}_n(\x)^{-1} = K^\mu_n(\x,\x) = \sum_{\balpha\in \N^d_n}P_{\balpha}(\x)^2, \quad\forall \x \in\R^d.
\end{equation}
It should be emphasized that the above identity holds for the nonnegative measure $\mu$, which is not necessarily 
a probability measure. 

In its variational characterization \eqref{eq:variational}, the Christoffel function is the optimal value of 
a quadratic convex optimization problem that can be solved efficiently. Its unique optimal solution $p^*\in\R[\x]_n$ reads:
\[
\x\mapsto p^*(\x):= \frac{\sum_{\balpha\in\N^d_n}P_{\balpha}(\bxi)P_{\balpha}(\x)}{\sum_{\balpha\in\N^d_n}P_{\balpha}(\bxi)^2}
= \frac{K^\mu_n(\bxi,\x)}{K^\mu_n(\bxi,\bxi)}\,,\quad\x\in\R^d.
\]
\begin{rem}
As $\Lambda^\mu_n$ and $K^\mu_n$ depend only on 
moments $\bmu=(\mu_{\balpha})_{\balpha\in\N^d_{2n}}$ of $\mu$, one may also define exactly in the same manner, the Christoffel function $\Lambda^{\bphi}_n$ and the CD-kernel $K^{\bphi}_n$ associated with a real (pseudo)-moment vector $\bphi\in\R^{s(2n)}$ such that $\M_n(\bphi)\succ0$.
\end{rem}
\subsection*{Equilibrium measure}

The notion of equilibrium measure associated with a given set originates from logarithmic potential theory 
(working with a compact set $E\subset\mathbb{C}$ in the univariate case). It minimizes the energy functional
\begin{equation}
\label{logarithmic}
I(\phi)\,:=\,\int\int \log{\frac{1}{\vert z-t\vert}}\,d\phi(z)\,d\phi(t)\,,
\end{equation}
over all Borel probability measures $\phi$ supported on $E$. For instance if $E$ is the interval $[-1,1]\subset\CC$ then the arcsine (or Chebyshev) distribution 
$\mu=dx/\pi\sqrt{1-x^2}$ is an optimal solution. It turns out that the integrand of 
\eqref{logarithmic} is also the Riesz  $s$-kernel 
\begin{equation}
\label{Riesz-s-kernel}
K_s(\x)\,:=\,\left\{\begin{array}{cl} \mathrm{sign}(s)\,\vert \x\vert^{-s}&
\mbox{if $-2< s<0$ or $s>0$}\\
-\log{\vert x\vert}&\mbox{if  $s=0$}\end{array}\right.\,,\quad \x\in\R^d \,,\:\x\neq0\,,
\end{equation}
for the couple $(d,s)=(1,2)$.

Some generalizations have been obtained in 
the multivariate case via pluripotential theory in $\mathbb{C}^d$. In particular, if $E\subset\R^d\subset\mathbb{C}^d$ 
is compact then its equilibrium measure (let us denote it by $\mu$) is equivalent to the Lebesgue measure 
on compact subsets of $\mathrm{int}(E)$. It has an even explicit expression if $E$ is convex and symmetric 
about the origin; see e.g. Bedford and Taylor \cite[Theorems 1.1 and 1.2]{bedford}. Several examples of sets $E$ 
with its equilibrium measure given in explicit form can be found in Baran \cite{Baran}. 
Importantly, the appropriate approach to define the (intrinsic) equilibrium measure $\mu$ of a compact subset of $\R^d$ with $d>1$, is to consider $\R^d$ as a subset of $\CC^d$ and invoke pluripotential theory with its tools from Complex analysis (in particular, plurisubharmonic functions (and their regularization) and Monge-Amp\`ere operator). 
In contrast to the one-dimensional  case $d=1$, in $\CC^d$ with $d>1$, there is no  kernel and so no ``energy" to minimize and the appropriate tool is the Monge-Amp\`ere differential operator.

However one may still define another notion of equilibrium measure now obtained by minimizing  over measures $\phi$ on $\R^d$ an integral functional $\hat{I}(\phi)$
similar to $I(\phi)$,  which  involves 
the Riesz $s$-kernel \eqref{Riesz-s-kernel} and an \emph{external field} $V$ (e.g. constant on the Euclidean ball $B(0,R)$ of $\R^d$ and infinite outside);
see e.g. \cite{Chafai,Dragnev,Saff-Totik}.
But of course the optimal solution $\phi^*$ now depends on the valid 
couples $(d,s)$ and working in $\R^d$ (and not in $\CC^d$),
the tools to obtain $\phi^*$  have nothing in common with pluripotential theory in $\CC^d$ to obtain $\mu$.
It is a coincidence\footnote{The authors want to thank N. Levenverg (in a private communication) for illuminating comments  on this point.} that 
for some values of $(d,s)$,  the equilibrium measure $\mu$ from pluripotential theory specialized to the Euclidean ball $E=B(0,R)\subset\R^d\,(\subset\CC^d)$ of radius $R$, coincides with the measure $\phi^*$ that minimizes $\hat{I}(\phi)$.
(Incidentally, the so-called Riesz potential associated to the Riesz $s$-kernel is also related to negative fractional powers of the Laplacian; see e.g.
\cite[p. 148--149]{Laplace}.)

For more details on equilibrium measures and pluripotential theory, the interested reader is referred to \cite{Baran,bedford,klimek}, the discussion in \cite[Section 6.8, p. 297]{Kirsch} as well as \cite[Appendix B]{Saff-Totik}, \cite{Levenberg-survey}, and the references therein.  In the sequel, when we speak about the equilibrium measure of  a compact subset $E\subset\R^d$, we refer to that in pluripotential theory (i.e., with $E$ considered as a subset of $\CC^d$).

Finally, it is also worth mentioning that Fekete points  (which provide very nice sets of points for polynomial interpolation) are strongly related to the equilibrium measure \cite[Theorem 4.5.1]{book}, and in particular  for the specific sets that we consider in this paper; see e.g.  \cite{Levenberg}.

\subsection*{Bernstein-Markov property}
A Borel measure $\phi$ supported 
on a compact set $\Omega \subset\R^d$ is said to satisfy the Bernstein-Markov property if there exists a sequence of 
positive numbers $(M_n)_{n\in\N}$ such that for all $n$ and $p\in\R[\x]_n$, 
\[\sup_{\x\in \Omega}\vert p(\x)\vert\,\leq\, M_n\cdot \left(\int_\Omega p^2\,d\phi\right)^{1/2}
\quad\mbox{and}\quad\lim_{n\to\infty}\log(M_n)/n\,=\,0;
\]
see e.g. \cite[Section 4.3.3]{book}. If a Borel measure $\phi$ on $\Omega$ has the Bernstein-Markov property,
then the sequence of measures $d\nu_n=\frac{d\phi(\x)}{s(n)\Lambda^\phi_n(\x)}$, $n\in\N$, converges to 
$\mu$ for the weak-$\star$ topology and, in particular,
\begin{equation}
 \label{weak-star}
 \lim_{n\to\infty}\int_\Omega\x^{\balpha}\,d\nu_n =\lim_{n \to\infty}\int_\Omega \frac{\x^{\balpha}\,d\phi(\x)}{s(t)\Lambda^\phi_n(\x)}
     = \int_\Omega \x^{\balpha}\,d\mu \,,\quad\forall \balpha\in\N^d\,
\end{equation}
(see e.g. \cite[Theorem 4.4.4]{book}). In addition, if a compact $\Omega \subset\R^d$ is regular
then $(\Omega,\mu)$ has the Bernstein-Markov property; see \cite[p. 59]{book}. 

\section{Background and motivation}
\setcounter{equation}{0}
In this section, we discuss a duality result and its connection with the convex optimization problems, which motivates our
study. 

Let $g_0:=\1$ be the constant polynomial equal to $1$, and let $G\subset\R[\x]$ be a finite set of polynomials that contains $g_0$. 
For every $g\in \R[\x]$, let $t_g:=\lceil\mathrm{deg}(g)/2\rceil$ and let  $Q_n(g)$ be the convex cone associated with $G$, defined by:
\begin{eqnarray} \label{quad-truncated}
    Q_n(g)&:=&\bigg\{\,\sum_{g\in G}\sigma_g\,g\,:\: \sigma_g\in \Sigma[\x]_{n-t_g}\,\bigg\}\,,\quad n\in\N\,.
\end{eqnarray}
Clearly, every polynomial $p\in Q_n(G)$ is nonnegative on the set 
\begin{equation} \label{set-S}
    \Omega\,:=\,\{\,\x\in\R^d: g(\x)\geq0\,,\:\forall g\in G\,\}\,.
\end{equation}
This is the reason that one may say that the sum-of-squares weights $(\sigma^*_g)_{g\in G}$, in the representation $p=\sum_{g\in G}\sigma^*_g\,g$ of $p$, provide $p$ with an algebraic certificate of its positivity on $\Omega$. 
Such a representation of $p$, however,  is not unique in general and, consequently, neither is such a positivity certificate of $p$
unique.

The associated dual convex cone $Q_n(G)^*\subset\R[\x]_{2n}^*$ of $Q_n(G)$ is defined by
\begin{equation}  \label{quad-truncated2}
    Q_n(g)^*:= \left \{\,\bphi\in\R^{s(2n)}\,:\: \M_{n-t_g}(g\cdot\bphi)\,\succeq\,0\,,\quad g\in G\,\right\}.
\end{equation}
\begin{lem}[\cite{cras,nesterov}]\label{lem1}
  If $p\in\mathrm{int}(Q_n(G))$ then there is some $\bphi\in \mathrm{int}(Q_n(G)^*)$ such that
   \begin{eqnarray} \label{duality-1}
        p(\x)&=&\sum_{g\in G}\v_{n-t_g}(\x)^T\M_{n-t_g}(g\cdot\bphi)^{-1}\,\v_{n-t_g}(\x)\,g(\x)\,,\quad\forall \x\in\R^d\\
        \label{duality-2}
        &=&\sum_{g\in G}\Lambda^{g\cdot\bphi}_{n-t_g}(\x)^{-1}\,g(\x)\,,\quad\forall \x\in\R^d\,.
        \end{eqnarray}
\end{lem}

Eq. \eqref{duality-1} is due to Nesterov \cite{nesterov} while \eqref{duality-2} is its interpretation in terms of 
Christoffel functions of appropriate linear functionals; see \cite{cras} for more details. Notice that 
\eqref{duality-1} provides $p$ with a \emph{distinguished} certificate of its positivity on the set
$\{\x\in\R^d: g(\x)\geq0\,,\:\forall g\in G\}$, stated in terms of a specific linear functional $\phi^*_{2n}\in\mathrm{int}(Q_n(G)^*)$.
An immediately question arises: \emph{What is the relation between $\phi^*_{2n}$ and $p$?} We will provide an answer 
for the case when $p=\1$ for certain sets $\Omega\subset\R^d$ with specific geometry.

Next, we consider the following two convex optimization problems. For every $n\in\N$ and 
a fixed subset $G_n\subseteq G$, the 
first problem states 
\begin{align} \label{primal}
\rho_n\,= \displaystyle{\inf_{\bphi\in\R^{s(2n)}}}  \bigg\{-\displaystyle{\sum_{g\in G_n}}& \log\mathrm{det}(\M_{n-t_g}(g\cdot \bphi)): \\
     & \quad \phi(1)\,=\,1, \,  \M_{n-t_g}(g\cdot\bphi)\,\succeq\,0\,,\:\forall g\in G_n\,\bigg\}\,, \notag
\end{align}
and the second problem states 
\begin{align} \label{dual}
 \rho^*_n\,= \displaystyle{ \sup_{\Q_g}} &\bigg \{\displaystyle\sum_{g\in G_n} \log\mathrm{det}(\Q_g)\,: \Q_g\succeq0\,,\forall g\in G_n, \\
 & \displaystyle{\sum_{g\in G_n}} s(n-t_g)= 
 \sum_{g\in G_n}g(\x)\cdot\v_{n-t_g}(\x)^T\Q_g\v_{n-t_g}(\x), \forall\x\in\R^d\,\bigg\} \notag
\end{align}
where the supremum is taken over real symmetric matrices $(\Q_g)_{g\in G_n}$ of respective sizes $s(n-t_g)$. (In Section 4 we describe how to choose $G_n\subset G$, $n\in\N$,  to obtain our main result.)
The two optimization problems \eqref{primal} and \eqref{dual}  are closely related. 

\begin{thm} \label{th-duality-p=1}
With $n\in\N$ and $G_n\subseteq G$ fixed, Problems \eqref{primal} and \eqref{dual} have the same finite optimal value
$\rho_n=\rho_n^*$ if and only if $\1\in \mathrm{int}(Q_n(G_n))$. Moreover, both have a unique optimal solution
 $\bphi^*_{2n}\in\R^{s(2n)}$ and $(\Q^*_g)_{g\in G_n}$ respectively, which satisfy 
$\Q_g^*=\M_{n-t_g}(g\cdot\bphi^*_{2n})^{-1}$ for all $g\in G_n$. And, as a consequence, 
\begin{align}  \label{eq:th1-1}
 1\, & =\frac{1}{ \displaystyle\sum_{g\in G_n} s(n-t_g)}\,
 \sum_{g\in G_n}g(\x)\,\v_{n-t_g}(\x)^T\M_{n-t_g}(g\cdot\bphi^*_{2n})^{-1}\v_{n-t_g}(\x)\\
  & = \frac{1}{ \displaystyle\sum_{g\in G_n} s(n-t_g)}\, \sum_{g\in G_n}g(\x)\,\Lambda^{g\cdot\bphi^*_{2n}}_{n-t_g}(\x)^{-1}\,,\quad\forall  \x\in\R^d. \notag
\end{align}
\end{thm}
Once again, one may interpret \eqref{eq:th1-1} as providing the constant polynomial $\1$ with a specific 
algebraic certificate of its positivity on the set $\Omega:=\{\,\x\in\R^d:\: g(\x)\,\geq\,0\,,\:\forall g\in G_n\,\}$. 
It also provides $\Omega$ with a polynomial \emph{partition of unity} in terms of its generators $g\in G_n$. Finally, 
it is worth noticing that both \eqref{primal} and \eqref{dual} are specific convex optimization problems with a ``$\log \mathrm{det}$" 
criterion, which can be solved with off-the-shelf mathematical softwares such as CVX \cite{cvx} or 
JuMP \cite{JuMP}.

An interesting and very specific case is when the (unique) optimal solution $\bphi^*_{2n}$ in Theorem \ref{th-duality-p=1}, 
for some $n_0\in\N$ and every $n\geq n_0$, is the vector of moments up to degree $2n$ of a same unique measure $\mu$ 
on a set $\Omega:=\{\x\in\R^d: g(\x)\geq 0\,,\:\forall g\in G\}$. In fact, $\Omega$ can be defined by a smaller set 
$\tilde{G}$ of generators while $G$ is made of some products of generators in $\tilde{G}$. This is because 
\eqref{eq:th1-1} may not hold with $\tilde{G}$ while it sometimes holds with a subset $G_n\subseteq G$;
however, for sufficiently large $n$, $G_n=G$. It then turns out 
that $\mu$ is necessarily the equilibrium measure of $\Omega$. 

\begin{ass}
\label{ass-1}
The set $\Omega$ in \eqref{set-S} is compact with nonempty interior. Moreover,
there exists $R>0$ such that the 
quadratic polynomial $\x\mapsto \theta(\x):=R-\Vert\x\Vert^2$ is an element of $Q_1(G)$.
In other words, $h\in Q_1(G)$ is an ``algebraic certificate" that $\Omega$ is compact.
\end{ass}

\begin{thm}
\label{th0}
With $\Omega$  as in \eqref{set-S}, let Assumption \ref{ass-1} hold. 
Let $\bphi=(\phi_{\balpha})_{\balpha\in\N^n}$ (with $\phi_0=1$) be such that $\M_n(g\cdot\bphi)\succ0$ for all $n\in\N$ and all $g\in G$, so that the Christoffel functions $\Lambda^{g\cdot\phi}_n$ are all well defined
(recall that $\phi\in\R[\x]^*$ is the Riesz  linear functional associated with the moment sequence
$\bphi$). In addition, suppose that there exists $n_0\in\N$ such that 
 \begin{equation}
 \label{th0-1}
 1\,=\,
 \frac{1}{\sum_{g\in G_n}s(n-t_g)}\,\sum_{g\in G_n} g\cdot(\Lambda^{g\cdot\phi}_{n-t_g})^{-1}\,,\quad\forall n\geq n_0\,,
 \end{equation}
 for some subset $G_n\subseteq G$ for all $n\geq n_0$ (with $G_n=G$ for $n$ sufficiently large). 
 Then $\phi$ is a Borel measure on $\Omega$ and it is the unique representing measure of $\bphi$. Moreover, if 
 $(\Omega,g\cdot\phi)$ satisfies the Bernstein-Markov property for every $g\in G$, then 
 $\phi=\mu$ (the equilibrium measure of $\Omega$) and  therefore the Christoffel polynomials $(\Lambda^{g\cdot\mu}_{n})^{-1}_{g\in G_n}$ satisfy the generalized Pell's equations 
\begin{equation}
\label{th0-2}
1\,=\,\frac{1}{\sum_{g\in G_n}s(n-t_g)}\,\sum_{g\in G_n} g\cdot(\Lambda^{g\cdot\mu}_{n-t_g})^{-1}\,,\quad\forall n\geq n_0\,.
\end{equation}
\end{thm}
The prototype example of Theorem \ref{th0} is the interval $\Omega:=[-1,1]$ with Chebyshev (equilibrium) measure
$\mu=dx/\pi\sqrt{1-x^2}$. Our main result in this paper is to identify three cases of sets $\Omega\subset\R^d$ for which 
this is precisely the case for \emph{all} $n$ and in \emph{any} dimension $d$, as suggested in \cite{cras} where they are verified only for 
$d = 2$ and $n = 1, 2, 3$. 

\begin{rem} \label{rem:3.1}
The equilibrium measure is a probability measure. In the identity  \eqref{th0-1}, however, the measure $g\cdot \mu$ is not 
normalized as a probability measure. This is important since the value of the Christoeffel function depends on the normalization. 
Indeed, if $\d \mu$ is a measure and $\gamma$ is a positive constant, then it is easy to verify that 
$$
   \Lambda_n^{\gamma \cdot \mu} (\x) = \gamma \Lambda_n^{\mu}(\x)\,.
$$
\end{rem}

\section{Main result}
\label{sec:main}
\setcounter{equation}{0}
In this section, we state our main results on \eqref{intro-0B} for several compact sets $\Omega$ with associated equilibrium 
measure denoted $\mu$. That is, we identify three cases of sets $\Omega\subset\R^d$  for which Theorem \ref{th0} holds 
in any dimension $d$. Since the properties of the orthogonal structure remain valid under the
affine transformation, we shall state our results only for regular domains. More precisely, we consider
\begin{description}
\item[Unit ball] $\Omega= \BB^d = \{\x\in\R^d: 1-\Vert\x\Vert^2\geq 0\}$ and $\mu$ is proportional to
 \begin{equation} \label{equi-ball}
    \frac{\d\x}{\sqrt{1-\Vert\x\Vert^2}}.
 \end{equation}
\item[Simplex]$\Omega=\triangle^d =\{\x\in\R^d_+: 1-\sum_{j=1}^d x_j\geq0\,\}$ and $\mu$ is proportional to
 \begin{equation} \label{equi-simplex}
       \frac{\d\x}{\sqrt{x_1}\cdots\sqrt{x_d}\sqrt{1-\sum_{j=1}^d x_j}}\,.
 \end{equation}
 \item[Unit cube] $\Omega=\square^d:=[-1,1]^d$ and $\mu$ is proportional to
 \begin{equation} \label{equi-box}
        \frac{\d\x}{\sqrt{1-x_1^2}\cdots\sqrt{1-x_d^2}}\,.
 \end{equation}
\end{description}
For more details (and even more examples of sets) the interested reader is referred to \cite{Baran}. Below we state and discuss our main result for each case in a subsection. The proof is postponed to Section \ref{proofs}.
\subsection{On the unit ball}
Let $\Omega =\BB^d$ with associated equilibrium measure $\mu$ as in \eqref{equi-ball}, 
and let $\x\mapsto g(\x):=1-\Vert\x\Vert^2$, for all $\x\in\R^d$.
\begin{thm} \label{th-ball}
Let $\mu$ be the equilibrium measure of $\BB^d$, normalized to be a probability measure. Then, for every $m\in\N$,
\begin{equation}\label{th-ball-1}
    P^\mu_m(\x,\x)+ g(\x)\,P^{g\cdot\mu}_{m-1}(\x,\x)= \binom{d+m-1}{d-1} + \binom{d+m-2}{d-1},
    \quad\forall \x\in\R^d.
\end{equation}
And as a consequence, for every $n\in\N$:
\begin{equation} \label{th-ball-2}
    \Lambda^\mu_n(\x)^{-1}+ g(\x)\,\Lambda^{g\cdot\mu}_{n-1}(\x)^{-1}\,=\binom{d+n}{d} + \binom{d+n-1}{d},\quad\forall\x\in\R^d\,.
\end{equation}
\end{thm}

As noted in Remark \ref{rem:3.1}, the weight function $g \cdot \mu$ is not normalized.
 
We note that the identity \eqref{th-ball-1} is the exact multivariate analog for the unit ball $\BB^d$ of the 
univariate polynomial Pell's equation \eqref{intro-0} for $[-1,1]$. It is now a property of all polynomials of degree $m$ 
orthonormal with respect to $\mu$ and all polynomials of degree $m-1$ orthonormal with respect to $g\cdot\mu$. Moreover,
the identity \eqref{th-ball-2} is a generalized Pell's equation satisfied by the Christoffel functions $\Lambda^\mu_n$ and 
$\Lambda^{g\cdot\mu}_n$. In other words, we have proved that in any dimension $d$,
the unit ball is an instance of a set $\Omega\subset\R^d$ for which
Theorem \ref{th0} holds. In this case $G_n=G=\{g\}$ for all $n$, with $\x\mapsto g(\x)=1-\Vert\x\Vert^2$.

\begin{cor}\label{cor-ball}
Let $\mu$ be the equilibrium measure of $\BB^d$ normalized to be a probability measure. Then, for every $n$,
\begin{equation} \label{cor-box-1}
\inf_{\x\in \BB^d} \Lambda^\mu_n(\x)\,=\,\min_{\x\in \BB^d} \Lambda^\mu_n(\x)\,=\,
\left[\binom{d+n}{d} + \binom{d+n-1}{d}\right]^{-1}\,=:\,\frac{1}{\gamma_n},
\end{equation}
and the minimum is attained at all points of the boundary $\sph$ of $\BB^d$. Moreover
\[\BB^d\,\subset\,\{\x: \Lambda^\mu_n(\x)\,\geq\,1/\gamma_n\,\},\quad \forall n\in\N\,.\]
\end{cor}
\begin{proof}
 By \eqref{th-ball-2},
\[
    \binom{d+n}{d} + \binom{d+n-1}{d} -\Lambda^\mu_n(\x)^{-1} = g(\x)\,\Lambda^{g\cdot\mu}_{n-1}(\x)^{-1}\,\geq\,0,
   \quad\forall \x\in \BB^d, 
\]
 and the right-hand side vanishes on $\sph$ because $g(\x)=0$ for all $\x\in \sph$.
\end{proof}

The above corollary states the polynomial sublevel set $\{\x: \Lambda^\mu_n(\x)^{-1}\leq \gamma_n\}$ is an 
outer-approximation of $\BB^d$ for all $n\in\N$. It also states that the polynomial $\x\mapsto \gamma_n-\Lambda^\mu_n(\x)^{-1}$ 
is in the ideal $\langle g\rangle$ generated by $g$.

\subsection{On the Simplex}
Let $\triangle^d=\{\x\in\R^d_+:\sum_{i=1}^dx_i\leq 1\}$, with associated equilibrium $\mu$ as in \eqref{equi-simplex}. 
With $\varepsilon\in\{0,1\}^{d+1}$ and $|\x| = \sum_{i=1}^dx_i$, define
\begin{equation}
    \label{g:simplex}
\x\mapsto g_{\varepsilon}(\x)\,:=\,x_1^{\varepsilon_1}\cdots x_d^{\varepsilon_d}\cdot (1-|\x|)^{\varepsilon_{d+1}}\,,\quad\forall \x\in\R^d,
\end{equation}
and introduce the vector $\1:=(1,1,\ldots,1)\in\N^{d+1}$.

\begin{thm}
\label{th-simplex}
Let $\mu$ be the equilibrium measure of $\triangle^d$, normalized to be a probability measure. Then for every $m\in\,2\N$ and $\x \in \RR^d$,
\begin{align}\label{th-simplex-1}
    \sum_{\varepsilon\in\{0,1\}^{d+1};\,|\varepsilon|\in 2\N;\,\vert\varepsilon\vert\leq 2m} 
              g_{\varepsilon}(\x)\,
        P^{g_\varepsilon\cdot\mu}_{m-|\varepsilon|/2}(\x,\x)
         = \binom{2n+d-1}{d-1} +  \binom{2n+d-2}{d-1}. 
\end{align}
And as a consequence, for every $n\in \N$ and $\x \in \RR^d$,
\begin{equation}\label{th-simplex-2}
    \sum_{\varepsilon\in\{0,1\}^{d+1};\,|\varepsilon| \in 2\N;\,\vert\varepsilon\vert\leq n}
      g_{\varepsilon}(\x)\,
        \Lambda^{g_{\varepsilon}\cdot\mu}_{n-\vert\varepsilon\vert/2}(\x)^{-1}\,= \binom{2n+d}{d},
\end{equation}
where $(a)_n = a(a+1) \cdots (a+n-1)$ is the usual Pochhammer symbol. 
\end{thm}

Again we have proved that in any dimension $d$, the simplex is an instance of a set $\Omega\subset\R^d$ for which 
Theorem \ref{th0} holds. Here $G=\{g_{\boldsymbol{\varepsilon}}:\boldsymbol{\varepsilon}\in\{0,1\}^d\}$ with
$g_{\boldsymbol{\varepsilon}}$ as in \eqref{g:simplex}, and $G_n=\{g_{\boldsymbol{\varepsilon}}\in G:
\boldsymbol{\varepsilon}\in 2\N\,;\:\vert\boldsymbol{\varepsilon}\vert\leq n \}$ so that
$G_n=G$ for sufficiently large $n$.

As an example, the identity \eqref{th-simplex-2} for $d =2$ is given by
\begin{align*}
 \Lambda_n^\mu(\x)^{-1} +  x_1 x_2   \Lambda_{n-1}^{x_1 x_2 \cdot \mu}(\x)^{-1}
    \,&  +  x_1 x_3   \Lambda_{n-1}^{x_1 x_3 \cdot \mu}(\x)^{-1}\\
       \,& + x_2 x_3   \Lambda_{n-1}^{x_2 x_3 \cdot \mu}(\x)^{-1}  = (n+1)(2n+1),
\end{align*}
where $\x = (x_1,x_2)$ and $x_3 = 1-x_1-x_2$. In this example, as well as in the Theorem \ref{th-simplex}, we should keep
in mind that $g_\varepsilon \cdot \mu$ is not normalized, as shown in Remark \ref{rem:3.1}.
 
We also have an analog of Corollary \ref{cor-ball} that follows from a similar proof. 

\begin{cor}
\label{cor-simplex}
Let $\mu$ be the equilibrium measure of $\triangle^d$, normalized to be a probability measure. Then, for every $n$,
\begin{equation}
\label{cor-simplex-1}
\inf_{\x\in \triangle^d} \Lambda^\mu_n(\x)\,=\,\min_{\x\in \triangle^d} \Lambda^\mu_n(\x)\,=\,
\left[\binom{2n+d}{d} \right]^{-1}\,=:\frac{1}{\gamma_n}, 
\end{equation}
and the minimum is attained at all points of the boundary $\partial \triangle^d$ of $\triangle^d$. Moreover
\[\triangle^d\,\subset\,\{\x: \Lambda^\mu_n(\x)\,\geq\,1/\gamma_n\,\}\,,\quad \forall n\in\N\,.\]
\end{cor}

Again, Corollary \ref{cor-simplex} states that the polynomial sublevel set $\{\x: \Lambda^\mu_n(\x)^{-1}\leq \gamma_n\}$
is an outer-approximation of the simplex $\triangle^d$ for all $n\in\N$. It also states that the polynomial 
$\x\mapsto \gamma_n-\Lambda^\mu_n(\x)^{-1}$ is in the ideal generated by all the polynomials $g_{\varepsilon}$ 
in \eqref{th-simplex-2}.

\subsection{The unit cube}
Let $\square^d =[-1,1]^d$ with associated equilibrium measure $\mu$ as in \eqref{equi-box}, and for every 
$\varepsilon\in\{0,1\}^d$, define
    \begin{equation}
        \label{g:cube}
    \x\mapsto g_{\varepsilon}(\x)\,:=\,\prod_{j=1}^d(1-x_j^2)^{\varepsilon_j}\,,\quad\forall \x\in\R^d\,.\end{equation}
 
\begin{thm} \label{th-box}
Let $\mu$ be the equilibrium measure of $\square^d$, normalized to be a probability measure. Then, for every $m \in\N$ and
$\x\in\R^d$,
\begin{equation} \label{th-box-1}
   \sum_{\varepsilon\in\{0,1\}^{d};\,\vert\varepsilon\vert\leq m}g_{\varepsilon}(\x)\,P^{g_\varepsilon\cdot\mu}_{m-\vert\varepsilon\vert}(\x,\x)     
             = \sum_{j=0}^d \binom{d}{j} \left[ \binom{d + n - j}{d} - \binom{d + n -1- j}{d}\right]
\end{equation}
and, as a consequence, for every $n\in \N$:
\begin{equation} \label{th-box-2}
    \sum_{\varepsilon\in\{0,1\}^{d};\,\vert\varepsilon\vert\leq n}g_{\varepsilon}(\x)\,
         \Lambda^{g_{\varepsilon}\cdot\mu}_{n-\vert\varepsilon\vert}(\x)^{-1}
             = \sum_{j=0}^d \binom{d}{j} \binom{d+n-j}{d},\quad\forall\x\in\R^d\,.
\end{equation}
\end{thm}

As in the case for the unit ball, the identity \eqref{th-box-1} is satisfied by all polynomials of degree $m-\varepsilon$, 
orthonormal with respect to the measures $g_{\varepsilon}\cdot\mu$, $\varepsilon\in\{0,1\}^d$ and $\vert\varepsilon\vert\leq m$. 
Moreover, the identity \eqref{th-box-2} is a generalized Pell's equation satisfied by the Christoffel functions 
$(\Lambda^{g_{\varepsilon}\cdot\mu}_{n-\vert\varepsilon\vert})_{\varepsilon\in\{0,1\}^d;\vert\varepsilon\vert\leq n}$.

Thus, we have proved that in any dimension $d$, the cube is also an instance of a set $\Omega\subset\R^d$ for
which Theorem \ref{th0} holds. Here $G=\{g_{\boldsymbol{\varepsilon}}:\boldsymbol{\varepsilon}\in\{0,1\}^d\}$ with
$g_{\boldsymbol{\varepsilon}}$ as in \eqref{g:cube}, and 
$G_n=\{\,g_{\boldsymbol{\varepsilon}}\in G: \vert\boldsymbol{\varepsilon}\vert\leq n\}$ so that $G_n=G$ as soon as $n\geq d$.

As an example, \eqref{th-box-2} with $d=2$ reads:
\begin{align*}
\Lambda^\mu_n(\x)^{-1} &\, + (1-x_1^2)\, \Lambda^{(1-x_1^2)\cdot\mu}_{n-1}(\x)^{-1}
   + (1-x_2^2) \,\Lambda^{(1-x_2^2)\cdot\mu}_{n-1}(\x)^{-1}\\
  & + (1-x_1^2)(1-x_2^2)\,\Lambda^{(1-x_1^2)(1-x_2^2)\cdot\mu}_{n-2}(\x)^{-1} = 1+2n(n+1).
\end{align*}
Thus, we have established the results initiated in \cite{cras} for \emph{all degrees $n$} and \emph{all dimensions $d$}, and 
we have demonstrated in the three cases listed above that, for every $n$, the Christoffel functions 
$(\Lambda^{g\cdot\mu}_n)_{g\in G}$ satisfy the generalized Pell's identity
\begin{equation}\label{generalized}
\sum_{g\in G_n}s(n-t_g)\,=\,\sum_{g\in G_n} g(\x)\,\Lambda^{g\cdot\mu}_n(\x)^{-1}\,,\quad\forall \x\in\R^d\,,\quad\forall n\in\N\,.
\end{equation}
It is important to emphasize that the set $G$ of generators is crucial. Indeed, in the case of the simplex for instance,
if we take $G=\{1,x_1,\ldots,x_d,1-\sum_jx_j\}$ then \eqref{th0-2} cannot hold. A similar conclusion
holds for the box with $G=\{1,1-x_1^2,\ldots,1-x_d^2\}$ as products of such generators are needed.

\subsection{Further examples}
As shown in the next section, our proof relies on analog to addition formulas on the domain,
which provides a powerful tool whenever it exists. Analogs of the addition formula are known 
to exist for orthogonal polynomials on hyperbolic surfaces and hyperboloids, but only for
the subclass of orthogonal polynomials that are even in the variable on the axis of rotation 
\cite{X21}. With an appropriate extension of our setup to the subclass of orthogonal 
polynomials, it is possible to establish analogs to our main results on the hyperbolic 
surfaces and hyperboloids. Since our main goal in this paper is to demonstrate, via 
examples, the remarkable multivariate analog and generalization of Pell's polynomial
equation, we decide not to include these further extensions as they are more technical 
in nature. 

\subsection{An extremal property of Christoffel functions}

We have established the results initiated in \cite{cras} for \emph{all degrees $n$} and \emph{all dimensions $d$}, and 
we have demonstrated in the three cases listed above that, for every $n$, the Christoffel functions 
$(\Lambda^{g\cdot\mu}_n)_{g\in G}$ satisfy the generalized Pell's identity \eqref{generalized}. 
Again it is important to emphasize that the set $G$ of generators is crucial. Indeed, recall that in the case of the box for instance,
if we take $G=\{1,1-x_1^2,\ldots,1-x_d^2\}$ then even though a unique optimal solution $\phi^*_{2n}$ of \eqref{primal} exists,
it cannot be the vector of moments up to degree $2n$ of the equilibrium measure $\mu$ of $\Omega$.

As the property \eqref{generalized} of Christoffel functions is quite strong, it is likely to hold only for sets $\Omega\subset\R^d$ 
with a very specific geometry and a very specific representation \eqref{set-S} since the choice $G$ of generators is also crucial. What about different sets $G$ of generators, or more general sets $\Omega\subset\R^d$?

In the rest of this section, we suppose that $G\subset\R[\x]$ is given and $\Omega\subset\R^d$ in \eqref{set-S} is compact. 
By Theorem \ref{th-duality-p=1}, if $\1\in \mathrm{int}(Q_n(G_n))$ for every $n\geq n_0$ and $G_n=G$ for sufficiently large $n$, 
there exists $\bphi^*_{2n}\in\R^{s(2n)}$ which is an optimal solution of \eqref{primal}, and
\begin{equation}
    \label{eq:general}
    \sum_{g\in G_n}s(n-t_g)\,=\,\sum_{g\in G_n} g(\x)\,\Lambda^{g\cdot\bphi^*_{2n}}_{n-t_g}(\x)^{-1}\,,
     \quad\forall \x\in\R^d\,,\quad\forall n\,   \geq\,n_0\,.
\end{equation}
The crucial difference with \eqref{generalized} is that now $\bphi^*_{2n}$ depends on $n$, whereas in \eqref{generalized}
the linear functionals $(g\cdot\mu)_{g\in G_n}$ do \emph{not} change. 
It may also happen that $\phi^*_{2n}$ does not have a representing measure. Notice, however, that \eqref{eq:general} still 
states an identity satisfied by Christoffel functions $\Lambda^{g\cdot\bphi^*_{2n}}_{n-t_g}$, $g\in G_n$, associated with the 
linear functionals $g\cdot\bphi^*_{2n}$, $g\in G_n$. Finally, we note that \eqref{eq:general} also provides a partition of unity for 
the set $\Omega$.

Interestingly, the convex optimization problem \eqref{primal} provides us with a tool to check whether $\bphi^*_{2n}$ is the moment vector of the equilibrium measure $\mu$.
Indeed if $\bphi^*_{2n}$ is the restriction to moments up to degree $2n$ of $\bphi^*_{2n+2}$ then indeed, $\bphi^*_{2n}$
may be the vector of moments of $\mu$, up to degree $2n$. This is very interesting because apart from sets $\Omega$
with special geometry (like in this paper), there is no simple characterization of the equilibrium measure $\mu$ (let alone 
numerical characterization).\\

Regarding the convex optimization problem \eqref{primal}, we mention the following assumption, in which $Q_1(G)$ is the 
cone defined in \eqref{quad-truncated}. 

\begin{ass}\label{ass:1}
The set $\Omega\subset\R^d$ in \eqref{set-S} is compact with a nonempty interior and the quadratic polynomial 
$\x\mapsto 1-\Vert\x\Vert^2$  belongs to $Q_1(G)$.
\end{ass}
It has been shown in \cite{cras} that under Assumption \ref{ass:1}, \eqref{primal} has always an optimal solution
$\bphi^*_{2n}\in\R^{s(2n)}$ and, in addition, the sequence $(\bphi^*_{2n})_{n\in\N}$ has accumulation points with associated 
converging subsequences. For each such subsequence $(n_k)_{k\in\N}$, 
\[
\lim_{k\to\infty}(\phi^*_{2(n_k)})_{\balpha}\,=\,\phi^*_{\balpha}\,,\quad\forall \balpha\in\N^d, 
\]
for some measure $\phi^*$ on $\Omega$. A natural question arises: \emph{is the measure $\phi^*$ unique and, if so, 
is $\phi^*$ related to the equilibrium measure of $\Omega$?} \\

Conversely, let $\mu$ be the equilibrium measure of the set $\Omega$. Then $\M_n(\mu)\succ0$ and $\M_{n-t_g}(g\cdot\mu)\succ0$ 
for every $g\in G$ and every $n\in\N$, $n\geq n_0:=\min_{g\in G} t_g$. Therefore, for every $n\geq n_0$,  the polynomial 
\begin{equation}\label{eq:weak}
    \x\mapsto p_n(\x)\,:=\,\frac{1}{\sum_{g\in G}s(n-t_g)}\,\sum_{g\in G}g(\x)\,\Lambda^{g\cdot\mu}_{n-t_g}(\x)^{-1}\,,\quad\forall \x\in\R^d\,,\end{equation}
is well-defined and belongs to $\mathrm{int}(Q_n(G))$. Let $d\mu_n:=p_n\,d\mu$, for every $n\geq n_0$. The measure 
$\mu_n$ is a probability measure on $\Omega$ and it is proved in \cite{cras} that
\[\lim_{n\to\infty}\int f\,d\mu_n\,=\,\int f\,d\mu\,,\quad\forall f\in\mathscr{C}(\Omega)\,, \]
that is, the sequence of probability measures $(\mu_n)_{n\in\N}$ converges weakly to $\mu$, denoted $\mu_n\Rightarrow\mu$ as $n\to\infty$. In other words, the density
$p_n$ behaves like the constant polynomial $\1$ when integrating continuous functions against $\mu_n$ 
(but the right-hand-side of \eqref{eq:weak} is not equal to $\1$). So it is fair to say that 
\[p_n\,\mu \,\Rightarrow\,\mu \quad(\mbox{with $p_n$ as in \eqref{eq:weak})}\]
 is
a \emph{weak} form of \eqref{eq:general} (and also provides a weak form of Theorem \ref{th0}). 

Next, under some additional assumption, we can indeed relate the linear functional $\bphi^*_{2n}$ which is the unique optimal solution of \eqref{primal} and 
the equilibrium measure $\mu$ of $\Omega$.

\begin{cor}\label{cor1}
  Let $\Omega\subset\R^d$ be compact with associated equilibrium measure $\mu=f_E(\x)d\x$ and suppose that 
  $\lim\limits_{n\to\infty}s(n)\Lambda^\mu_n(\x) =1$ uniformly on compact subsets of $\mathrm{int}(\Omega)$. Assume that 
  $g(\x)>0$ for all $\x\in\mathrm{int}(\Omega)$ and all $g\in G$, and let $p_n$ be as in \eqref{eq:weak}. Then
  \begin{equation} \label{cor1:1}
        \lim_{n\to\infty}p_n(\x)=1\,,\quad \mbox{uniformly on compact subsets of $\mathrm{int}(\Omega)$.}
    \end{equation}
  In addition, assume that $\1\in \mathrm{int}(Q_n(G_n))$ for all $n\geq$  and let $\bphi^*_{2n}$ be the unique optimal solution of 
  \eqref{primal}. Then, uniformly on compact subsets of $\Omega$,
  \begin{equation}  \label{cor1-2}
      \lim_{n\to\infty}g(\x)\,\frac{\Lambda^{g\cdot\bphi^*_{2n}}(\x)^{-1}}{s(n-t_g)}\,=\,
      \lim_{n\to\infty}g(\x)\,\frac{\Lambda^{g\cdot\mu}(\x)^{-1}}{s(n-t_g)}\,=\,
      1\,,\quad\forall g\in G\,.
  \end{equation}
\end{cor}
\begin{proof}
 As $g>0$ on $\mathrm{int}(\Omega)$, by a result of Kro\'o and Lubinsky \cite{Lubinsky} (see also \cite[Theorem 4.4.1, p. 52]{book}), 
 and using our assumption on $\mu$,
 \[\lim_{n\to\infty} s(n)\,\Lambda_n^{g\cdot\mu}(\x)\,=\,g(\x)
 \]
uniformly on compact subsets of $\mathrm{int}(\Omega)$. Rewriting \eqref{eq:weak} as
\[p_n(\x)\,=\,\sum_{g\in G}\frac{g(\x)}{s(n-t_g)\,\Lambda^{g\cdot\mu}_{n-t_g}(\x)}\cdot
\frac{s(n-t_g)}{\sum_g s(n-t_g)}\,,\quad\forall \x\in\R^d\,,\]
and taking limit as $n$ grows when $\x\in\mathrm{int}(\Omega)$, we obtain \eqref{cor1:1}.

Next, let $\bphi^*_{2n}$ be an optimal solution of \eqref{primal}. 
Let $\Delta\subset\mathrm{int}(\Omega)$ be an arbitrary compact subset, and with $\varepsilon>0$ fixed, let $n$ be large enough so that $G_n=G$ and

$\bullet$  $(\sum_{g\in G}s(n-t_g))/s(n-t_g)\leq 1+\varepsilon$ for all $g\in G$, 

$\bullet$ $1-\varepsilon\leq g(\x)\frac{\Lambda^{g\cdot\mu}_{n-t_g}(\x)^{-1}}{\sum_{g\in G}s(n-t_g)}\leq 1+\varepsilon$. 

\noindent
Then by \eqref{eq:general},
$0\leq g(\x)\Lambda^{g\cdot\bphi^*_{2n}}_n(\x)^{-1}/s(n-t_g)\leq (1+\varepsilon)$, for all $g\in G$, and all $\x\in\Delta$, so that
for every $\x\in \Delta$ and every $g\in G$,
\[-2\varepsilon\,\leq\,\frac{g(\x)\,(\Lambda^{g\cdot\bphi^*_{2n}}_{n-t_g}(\x)^{-1}-\Lambda^{g\cdot\mu}_{2n}(\x)^{-1})}{s(n-t_g)}\,\leq\,2\varepsilon\,.\]
As $\varepsilon>0$ was arbitrary, one obtains
\[\lim_{n\to\infty}\,g(\x)\,\frac{\Lambda^{g\cdot\bphi^*_{2n}}_{n-t_g}(\x)^{-1}}{s(n-t_g)}\,=\,\lim_{n\to\infty}\,
g(\x)\,\frac{\Lambda^{g\cdot\mu}_{n-t_g}(\x)^{-1}}{s(n-t_g)}\,=\,1\,,\quad\forall \x\in\Delta\,,\:\forall g\in G\,.\]
\end{proof}
So, under the assumptions in Corollary \ref{cor1}, the linear functional $\bphi^*_{2n}$, unique optimal solution of \eqref{primal}, 
behaves asymptotically like the equilibrium measure of $\Omega$, as $n$ grows. Thus, for instance, for the unit cube with the set of generators 
$G=\{1,1-x_1^2,\ldots,1-x_d^2\}$, $\bphi^*_{2n}$ cannot be the vector of moments up to degree $2n$ of the equilibrium
 measure $\mu$. However:
\[    \lim_{n\to\infty}\frac{\Lambda^{\bphi^*_{2n}}_n(\x)^{-1}}{s(n)}
    \,=\,\lim_{n\to\infty}\frac{\Lambda^{\mu}_n(\x)^{-1}}{s(n)}\,=\,1\,,\quad\forall \x\in\mathrm{int}([-1,1]^d)\,,\]
    and for every $i=1,\ldots,d$,
 \[\lim_{n\to\infty}(1-x_i^2)\,\frac{\Lambda^{(1-x_i^2)\cdot\bphi^*_{2n}}_{n-1}(\x)^{-1}}{s(n-1)}\,=\,\lim_{n\to\infty}(1-x_i^2)\,\frac{\Lambda^{(1-x_i^2)\cdot\mu}_{n-1}(\x)^{-1}}{s(n-1)}\,=\,1\,,\]
 for all $\x\in\mathrm{int}([-1,1]^d)$.

\section{Proofs} \label{proofs}
\setcounter{equation}{0}
Our proof depends heavily on orthogonal polynomials on a family of weight functions associated with the Chebyshev
weight. To avoid confusion, we shall adopt notations that make it transparent how kernels depend on weight 
functions. More precisely, let $W$ be a weight function on a domain $\Omega$; we consider orthogonal polynomials 
with respect to the inner product
$$
     \la f, g\ra  = \int_{\Omega} f(\x) g(\x) W(\x)\, \d \x\,.
$$
We then denote the space $\CV_n^d$ of orthogonal polynomials of degree $n$ by $\CV_n(W; \Omega)$. Then 
$$
    \dim \Pi_n^d  = \binom{n+d}{n}\quad \hbox{and}\quad \dim \CV_n(W; \Omega) = \binom{n+d-1}{n}.
$$
Let $P_n(W; \cdot,\cdot)$ be the reproducing kernel of $\CV_n(W,\Omega)$. If $\{P_\alpha^n: |\alpha| = n, \alpha \in \NN_0^d\}$
is an orthonormal basis of $\CV_n(W, \Omega)$, then the reproducing kernel of $\CV_n(W,\Omega)$ is 
$$
  P_n(W; \x, \y) = \sum_{|\alpha| = n} P_\alpha^n(\x) P_\alpha^n(\y). 
$$
Moreover, the reproducing kernel of $\Pi_n^d$ in $L^2(W, \Omega)$ is denoted by 
$$
  K_n(W; \x,\y) = \sum_{k=0}^n P_k(W; \x, \y). 
$$

\subsection{Unit ball}
The proof of Theorem \ref{th-ball} relies on the addition formula \eqref{intro:addtion} for the spherical harmonics 
mentioned in the introduction. Spherical harmonics are the restrictions of harmonic polynomials on the unit sphere and 
they are known to be orthogonal on the unit sphere. Let $\CH_n^{d+1}$ be the space of spherical harmonics of degree 
$n$ in $d + 1$ variables. For $n \in \NN_0$, let $\{Y_\ell^n: 1 \le \ell \dim \CH_n^{d+1}\}$ be an orthonormal basis 
of $\CH_n^{d+1}$, so that 
$$
    \frac{1}{\s_d} \int_{\SS^d} Y_{\ell}^n (\bxi) Y_{\ell'}^{m}(\bxi)\, \d \s(\bxi) = \delta_{\ell, \ell'} \delta_{n,m},
$$ 
where $\s_d$ denote the surface area of $\SS^d$ and $\d \s$ denote the Lebesgue measure 
on $\SS^d$. We denote by $P_n(\d\s; \cdot,\cdot)$ the reproducing kernel of $\CH_n^{d+1}$, which then satisfies 
the {\it addition formula}
\begin{equation} \label{eq:additionSph}
   Y_n(\d\s; \x,\y) = \sum_{\ell=1}^{\dim \CH_n^{d+1}} Y_{\ell}^n(\x) Y_\ell^n(\y) = Z_n^{\frac{d-1}2}(\la \x, \y\ra), 
   \quad \x, \y \in \SS^d,
\end{equation}
where $Z_n^\l$ is a multiple of the Gegenbauer (ultraspherical) polynomial, 
$$
  Z_n^\l(t)  = \frac{n+\l}{\l} C_n^{\l}(t). 
$$
In particular, setting $\y = \x$ in the addition formula, we obtain
\begin{equation} \label{eq:PellSph}
    \sum_{\ell} |Y_{\ell}^n(\x)|^2 = \frac{2n+d-1}{d-1} C_n^{\frac{d-1}{2}}(1) =  \binom{n+d-1}{d-1} +  \binom{n+d-2}{d-1},
\end{equation}
which follows from $C_n^\l(1) = \Gamma(n+2\l)/(n! \Gamma(2\l))$ and a simple verification.  

Taking a cue of \eqref{eq:additionSph}, for any domain $\Omega$ with weight function $W$, an addition 
formula for the reproducing kernel of $\CV_n(W; \Omega)$ is a closed-form formula 
for $P_n(W; \cdot,\cdot)$. 

We now turn our attention to the unit ball $\BB^d$. The classical orthogonal polynomials on $\BB^d$ are associated 
with the weight function $ (1-\|\x\|^2)^{\alpha}$, $\alpha > -1$. The normalization constant $c_\alpha$ of this weight 
function, so that $c_\alpha (1-\|x\|^2)^\alpha$ is a probability measure, is given by
$$
  c_\alpha = \frac{\Gamma(\alpha + \frac{d+2}{2})} {\pi^{\f{d}{2}} \Gamma(\alpha+1)}. 
$$ 
In the following, we denote by
$$
 W_{-\f12}(\x) = c_{-\f12} (1-\|\x\|^2)^{-\f12} \quad \hbox{and} \quad  W_{\f12}(\x) = c_{-\f12} (1-\|\x\|^2)^{\f12}. 
$$
Notice that $W_{-\f12}$ is a probability measure on $\BB^d$, but  $W_{\f12}$ is not. The spherical harmonics are closely
 related to orthogonal polynomials associated with $W_{-\f12}$ and $W_{\f12}$ on the unit ball.  In particular, we have the 
 following addition formula: 

\begin{prop}
Let $X= (\x,\sqrt{1-\|\x\|^2})$, $\x \in \BB^d$, and $\x \in \SS^{d+1}$. Then
\begin{equation}\label{eq:additionBall}
   Y_n(\d\s; X,Y) =  P_n\left(W_{-\f12}; \x, \y\right)+ \sqrt{1-\|\x\|^2} \sqrt{1-\|\y\|^2} P_{n-1}\left(W_{\f12}; \x, \y\right).
\end{equation}
\end{prop}

\begin{proof}
Let $\{P_\nu^n(W_\mu; \x): |\nu| = n, \nu \in \NN_0^d\}$ denote an orthogonal basis of $\CV_n^d(W_\mu)$. 
By \cite[Theorem 4.2.4]{DX}, it follows that an orthogonal basis for the space $\CH_n^{d+1}$ of spherical harmonics 
consists of 
\begin{align*}
  Y_\nu^{(1)} (\x, x_{d+1}) \, &= P_\nu^n(W_{-\f12},\x), \quad |\nu|=n,  \\
    Y_\nu^{(2)} (\x, x_{d+1}) \, & = x_{d+1} P_\nu^{n-1}(W_{\f12},\x), \quad |\nu|=n-1, 
\end{align*}
where $\x \in \BB^d$, $(\x,x_{d+1}) \in \SS^d$. Using the integral identity \cite[Lemma 4.2.3]{DX}, 
$$
  \int_{\SS^d} f(\y) \d \s(\y) = \int_{\BB^d} \left[ f\left(\x,\sqrt{1-\|\x\|^2}\right)+ f\left(\x,\sqrt{1-\|\x\|^2}\right)\right] \frac{\d \x}{\sqrt{1-\|\x\|^2}},
$$
it follows readily that $c_{-\f12}/2 = 1/ \omega_d$ is the surface area of $\SS^d$ and, consequently, 
\begin{align*}
  \| Y_\nu^{(1)}\|_{L^2(\d \s)}^2 = \|P_\nu^n(W_{-\f12})\|_{L^2(W_{-\f12})}^2  
\end{align*}
Moreover, we also obtain 
\begin{align*}
   \| Y_\nu^{(2)}\|_{L^2(\d \s)}^2 \,& = \frac{1}{\omega_d} \int_{\SS^d}  \left |x_{d+1} P_\nu^{n-1}(W_{\f12},\x)\right|^2 \d \s \\
   & = c_{-\f12} \int_{\BB^d} \left|P_\nu^{n-1}(W_{\f12},\x)\right|^2 \sqrt{1-\|\x\|^2} \d \x =
      \left\|P_\nu^{n-1}(W_{\f12})\right\|_{L^2(W_{\f12})}^2. 
\end{align*}
Consequently, since the reproducing kernel can be written as the sum of products of orthonormal polynomials, 
it follows that $ Y_n(\d\s; X,Y)$ is equal to 
\begin{align*}
  &   \sum_{|\nu|=n} \frac{ P_\nu^{n}(W_{-\f12},\x)P_\nu^{n}(W_{-\f12},\y)}{ \|P_\nu^n(W_{-\f12})\|_{L^2(W_{-\f12})}^2}
     + \sum_{|\nu|=n-1} \frac{ x_{d+1} y_{d+1} P_\nu^{n-1}(W_{\f12},\x)P_\nu^{n-1}(W_{\f12},\y)}{  \|P_\nu^{n-1}(W_{\f12})\|_{L^2(W_{\f12})}^2}\\
   &  =  P_n\left(W_{-\f12}; \x, \y\right)+  \sqrt{1-\|\x\|^2} \sqrt{1-\|\y\|^2} P_{n-1}\left(W_{\f12}; \x, \y\right).
\end{align*}
This completes the proof. 
\end{proof}

\begin{rem}
As an illustration for the Remark \ref{rem:3.1}, we note that if we replace $W_{\f12}$ by the probability measure 
$\wh W_{\f12}(x)= c_{\f12} (1-\|x\|^2)^{\f12}$, then 
\begin{align*}
   \| Y_\nu^{(2)}\|_{L^2(\d \s)}^2 \,& = \frac{1}{\omega_d} \int_{\SS^d}  \left |x_{d+1} P_\nu^{n-1}(\wh W_{\f12},\x)\right|^2 \d \s \\
   & = \frac{c_{-\f12}}{c_{\f12}}  \int_{\BB^d} \left|P_\nu^{n-1}(\wh W_{\f12},\x)\right|^2 \wh W_{\f 12}(x) \d \x =
     \frac{1}{d+1}  \left\|P_\nu^{n-1}(\wh W_{\f12})\right\|_{L^2(\wh W_{\f12})}^2.
\end{align*}
Hence, following the proof of the above proposition, we obtain the identity 
\begin{equation*}
   Y_n(\d\s; X,Y) =  P_n\left(W_{-\f12}; \x, \y\right)+ (d+1) \sqrt{1-\|\x\|^2} \sqrt{1-\|\y\|^2} P_{n-1}\left(\wh W_{\f12}; \x, \y\right).
\end{equation*}
Notice the additional $(d+1)$ in this identity in comparison with that of \eqref{eq:additionBall}
\end{rem}

Setting $Y =X$ in the formula that we just proved and applying \eqref{eq:PellSph} leads to the following
Pell identity on the unit ball.

\begin{cor}
For $d \ge 1$, $n=0,1,2,\ldots$ and $\x \in \BB^d$, 
\begin{equation} \label{eq:reprod:ball}
 P_n\left(W_{-\f12}; \x, \x\right)+ (1-\|\x\|^2) P_{n-1}\left(W_{\f12}; \x, \x\right) =  \binom{n+d-1}{d-1} +  \binom{n+d-2}{d-1} 
 \end{equation}
and, summing up the identity, 
\begin{equation} \label{eq:reprod2:ball}
 K_n\left(W_{-\f12}; \x, \x\right)+ (1-\|x\|^2) K_{n-1}\left(W_{\f12}; \x, \x\right) 
     =  \binom{n+d}{d} +  \binom{n+d-1}{d}\,.
\end{equation}
\end{cor}

The identity \eqref{eq:reprod:ball} is exactly the same as \eqref{th-ball-1} in Theorem \ref{th-ball}, whereas \eqref{eq:reprod2:ball}
is equivalent to \eqref{th-ball-2} by \eqref{ortho-poly-0}.

\subsection{Simplex}
The classical orthogonal polynomials on $\triangle^d$ are associated with the weight function 
$$
  W_{\bka}(\x) =  x_1^{\k_1} \cdots x_d^{\k_d} (1-|\x|)^{\k_{d+1}}, \qquad \x \in \triangle^d, \quad \k_i > -1. 
$$
The normalization constant $c_{\bka}$ that makes this weight function a probability measure is given by
$$
  c_{\bka} = \frac{\Gamma(|\bka|+d+1)}{\Gamma(\k_1+1) \cdots \Gamma(\k_{d+1}+1)}.
$$ 
For $\one = (1,\ldots,1) \in \RR^{d+1}$, we define the modified weight function 
$$
 W_{\bka}^\triangle(x) = c_{-\f \one 2} W_{\bka}(x).
$$
By this definition, the weight function $W_{-\one/2}$ is normalized to have the unit integral on $\triangle^d$ but $W_{\bka}$
for $\ve \ne -  \one/ 2$ is not. 

The classical orthogonal polynomials for $W_{\bka}^\triangle$ and those on the ball are closely related, as seen in \cite[Theorem 4.4.4]{DX}. Also, the following lemma is known in principle, but we need its precise form and the norm identities.

 \begin{lem}\label{lem:simplex}
 Let $\{P_{\boldsymbol{\nu}}^{\bka}: |\boldsymbol{\nu}| = n\}$ denote an orthogonal basis for $\CV_n(W_{\bka}^\triangle; \triangle^d)$.
 Let $\ve \in \{0,1\}^d$, and define, with $\bka = (\bka', \k_{d+1})$ and $\bka' \in \RR^d$, 
 $$
  \x\mapsto Q_{\boldsymbol{\nu},\ve}^{\bka} (\x) \,:=\, \x^\ve\, P_{\boldsymbol{\nu}}^{{\bka'+\ve,\k_{d+1}}} \left(x_1^2, \ldots, x_d^2\right), 
      \quad \boldsymbol{\nu} \in \NN_0^d.
$$
Let $\mathbf{1} = (1,\ldots,1) \in \RR^{d}$. Then 
$\{Q_{\boldsymbol{\nu},\ve}^{(-\mathbf{1}/2, \alpha)}: \ve \in \{0,1\}^d, \, |\boldsymbol{\nu}| =
 \frac{n-|\ve|}{2} \in \NN_0, \, \boldsymbol{\nu} \in \NN_0^d\}$ is an orthogonal basis for $\CV_n(W_{\alpha}, \BB^d)$ and,
 moreover,
 $$
   \left \| Q_{\boldsymbol{\nu},\ve}^{ (-\mathbf{1}/2, \alpha)} \right\|_{L^2(W_\alpha^\BB)}^2 = 
           \left \| P_{\boldsymbol{\nu}}^{{ (-\mathbf{1}/2+\ve,\alpha)}}  \right\|_{L^2(W_{-\mathbf{1}/2+\ve,\alpha}^\triangle)}^2.
 $$
\end{lem}

\begin{proof}
For $\ve \ne \ve'$, the orthogonality of $Q_{\boldsymbol{\nu},\ve}^{(-\mathbf{1}/2, \alpha)}$ and 
$Q_{\boldsymbol{\nu}',\ve'}^{- (\mathbf{1}/2, \alpha)}$ follows from the parity and the invariance of $W_\alpha$ under 
sign changes. For fixed $\ve$, the orthogonality of 
$\{Q_{\boldsymbol{\nu},\ve}^{- (\mathbf{1}/2, \alpha)}: \boldsymbol{\nu} \in \NN_0^d\}$ follows from \cite[Theorem 4.4.4]{DX}.
Furthermore, it is not difficult to verify that the cardinality of $\{Q_{\boldsymbol{\nu},\ve}^{- (\mathbf{1}/2, \alpha)}: \ve \in \{0,1\}^d, \, |\boldsymbol{\nu}| = \frac{n-|\ve|}{2} \in \NN_0, 
\, \boldsymbol{\nu} \in \NN_0^d\}$ is equal to $\dim \CV_n(W_\alpha; \BB^d)$, so that the set is an orthogonal basis. Moreover,
using the integral identity (\cite[Lemma 4.4.1]{DX})
$$
  \int_{\BB^d} f(y_1^2,\ldots, y_d^2) \d y = \int_{\triangle^d} f(x_1,\ldots, x_d) \frac{\d x }{\sqrt{x_1 \cdots x_d}},
$$
it is easy to verify that $c_{-\one/2, \alpha}$ is the normalization constant of $W_\alpha^\BB$ on the unit ball. Hence,
\begin{align*}
   \left \| Q_{\boldsymbol{\nu},\ve}^{ (-\mathbf{1}/2, \alpha)} \right\|_{L^2(W_\alpha^\BB)}^2 
  \,& = c_{-\one/2, \alpha}  \int_{\BB^d} \left| \x^\ve\, P_{\boldsymbol{\nu}}^{(- \one /2+\ve,\alpha)} \left(x_1^2, \ldots, x_d^2\right) \right|^2
   W_\alpha^\BB(x) \d x \\
  \, &  = c_{-\one/2, \alpha}  \int_{\TT^d} \left| P_{\boldsymbol{\nu}}^{(- \one /2+\ve,\alpha)}(x_1, \ldots, x_d) \right|^2
     W_{-\one/2+\ve, \alpha}^\triangle(x) \d x \\
  \, & =
    \left \| P_{\boldsymbol{\nu}}^{{ (-\mathbf{1}/2+\ve,\alpha)}}  \right\|_{L^2(W_{-\mathbf{1}/2+\ve,\alpha}^\triangle)}^2,
\end{align*}
where the last step follows from our definition of $W_{\bka}^{\vartriangle}$. This completes the proof. 
\end{proof}

\begin{prop}
For $\k_1,\ldots, k_{d+1} > -1$, $\x\in \triangle^d$, 
\begin{align}\label{eq:reprod:simplex}
   \sum_{\ve \in \{0,1\}^{d+1}, \, |\ve| \in 2\NN_0} 
      P_{n - \frac{|\ve|}{2}}\left(W_{-\mathbf{1}/2 + \ve}^\triangle; \x,\x\right) 
      =   \binom{2n+d-1}{d-1} +  \binom{2n+d-2}{d-1}. 
\end{align} 
Moreover, summing up the identity, 
\begin{align}\label{eq:reprod:simplex2}
   \sum_{\ve \in \{0,1\}^{d+1}, \, |\ve| \in 2\NN_0}
         \x^{\ve} K_{n - \frac{|\ve|}{2}}\left(W_{-\mathbf{1}/2 + \ve}^\triangle; \x,\x\right)
   =    \binom{2n+d}{2n}.
\end{align}
\end{prop}
 
\begin{proof}
Let $\x^2=(x_1^2,\ldots,x_d^2)$. As a consequence of the lemma, we obtain immediately 
\begin{align*}
  P_n(W_{\alpha}^\BB; \x,\y)\, & =  \sum_{\substack{\ve \in \{0,1\}^d \\ \frac{n-|\ve|}2 \in \NN_0}}
   \sum_{ |\boldsymbol{\nu}| =\frac{n-|\ve|}{2} }
    \frac{Q_{\boldsymbol{\nu},\ve}^{ (-\mathbf{1}/2, \alpha)}(\x) Q_{\boldsymbol{\nu},\ve}^{ (-\mathbf{1}/2, \alpha)}(\y)}
     {   \left \| Q_{\boldsymbol{\nu},\ve}^{ (-\mathbf{1}/2, \alpha)} \right\|_{L^2(W_\alpha^\BB)}^2} \\ 
      & =  \sum_{\substack{\ve \in \{0,1\}^d \\ \frac{n-|\ve|}2 \in \NN_0}} 
           \x^\ve \y^\ve
   \sum_{ |\boldsymbol{\nu}| =\frac{n-|\ve|}{2} }
    \frac{ P_{\boldsymbol{\nu}}^{{ (-\mathbf{1}/2+\ve,\alpha)}}(\x^2)  P_{\boldsymbol{\nu}}^{{ (-\mathbf{1}/2+\ve,\alpha)}}(\y^2)}
     { \left \| P_{\boldsymbol{\nu}}^{{ (-\mathbf{1}/2+\ve,\alpha)}}  \right\|_{L^2(W_{-\mathbf{1}/2+\ve,\alpha}^\triangle)}^2} \\ 
   & =  \sum_{\substack{\ve \in \{0,1\}^d \\ \frac{n-|\ve|}2 \in \NN_0}}  
     \x^\ve \y^\ve P_\frac{n-|\ve|}{2}\left(W_{-\mathbf{1}/2+\ve, \alpha}^\triangle; \x^2,\y^2\right), 
\end{align*}
where the sum is over those $\ve \in \{0,1\}^d$ such that $\frac{n-|\ve|}2 \in \NN_0$, which implies, in particular,
$$
  P_n(W_{\alpha}^\BB; \x,\x) =  \sum_{\substack{\ve \in \{0,1\}^d \\ \frac{n-|\ve|}2 \in \NN_0}} 
        \x^{2\ve} P_\frac{n-|\ve|}{2} \left(W_{-\mathbf{1}/2+\ve, \alpha}^\triangle; \x^2,\x^2\right).
$$
Choosing $\alpha = \pm \f12$ and applying the relation \eqref{eq:reprod:ball}, we obtain 
\begin{align*}
&  Z_n^{\frac{d-1}{2}}(1) =
  \sum_{\substack{\ve \in \{0,1\}^d \\ \frac{n-|\ve|}2 \in \NN_0}}  
        \x^{2\ve} P_\frac{n-|\ve|}{2}
      \left(W_{-\mathbf{1}/2+\ve, -\f12}^\triangle; \x^2,\x^2\right) \\
  & +   (1-\|\x\|^2) \sum_{\substack{\ve \in \{0,1\}^d \\ \frac{n-1-|\ve|}2 \in \NN_0}}
       \x^{2\ve} P_\frac{n-|\ve|-1}{2}
   \left(W_{-\mathbf{1}/2+\ve, \f12}^\triangle; \x^2,\x^2\right). 
\end{align*}
The two sums on the right-hand side can be combined by 
setting $\ve \in \{0,1\}^{d+1}$. Thus, changing variables $\x \mapsto (\sqrt{x_1},\ldots, \sqrt{x_d})$ and setting $x_{d+1} = 1-|\x|$, 
we obtain
\begin{align*}
   \sum_{\substack{\ve \in \{0,1\}^{d+1} \\ \frac{n-|\ve|}2 \in \NN_0}}
      P_{\frac{n-|\ve|}{2}}\left(W_{-\mathbf{1}/2 + \ve}^\triangle; \x,\x\right)
   =  Z_n^{\frac{d-1}2}(1)
\end{align*}
by \eqref{eq:PellSph}. In particular, setting $n \mapsto  2n$, we have proved \eqref{eq:reprod:simplex}. Summing up
this identity gives \eqref{eq:reprod:simplex2}, as can be seen be verifying the sum over binomial coefficients. 
\end{proof}

Like the case for the unit ball, \eqref{eq:reprod:simplex} is the same as \eqref{th-simplex-1}, whereas \eqref{eq:reprod:simplex2}
is equivalent to \eqref{th-simplex-2} by \eqref{ortho-poly-0}.

\subsection{Cube}
The classical orthogonal polynomials on $[-1,1]^d$ are associated with the weight function 
$$
  W_{\boldsymbol{\alpha}}^\square(\x) = \prod_{i=1}^d (1-x_i^2)^{\alpha_i}, \qquad \alpha_i > -1,
$$
and they are products of univariate Gegenbauer polynomials. More precisely, an orthonormal basis for 
$\CV_n(\boldsymbol{\alpha}^{\square}; [-1,1]^d)$ is given by 
$$
   P_\kb\left(W_{\boldsymbol{\alpha}}^\square; \x\right) = \prod_{j=1}^d \wh C_{k_j}^{\alpha_j+\f12} (x_j), \qquad |\kb| = n, \quad \kb \in \NN_0^d,
$$
where $\wh C_n^{\lambda}$ denotes the orthonormal Gegenbauer polynomial of degree $n$. It is worth mentioning that 
orthonormality is defined with respect to the probability measure. In particular, normalized Chebyshev weight 
functions are $\frac{1}{\pi} (1-x^2)^{-\f12}$ and $\frac{1}{2 \pi} (1-x^2)^{\f12}$, so that the normalized Chebyshev 
polynomials $\wh T_n$ of the first kind are 
$$
  \wh T_n(x) = \sqrt{\delta(k)} \,T_n(x), \quad n \ge 0, \quad \hbox{where} \quad \delta(k) = \begin{cases}1 & k=0, \\ 2 & k >0. \end{cases}
$$
whereas the Chebyshev polynomials $U_n$ of the second kind are already orthonormal.  
 
A closed-form formula for the reproducing kernel for $\CV_n(W_{- \mathbf{1} /2}^{\square}, [-1,1]^d)$ was derived in \cite{X95} 
and \cite{BX}, and further explored in \cite{BJX} recently. It shows, in particular, that if $\t = (\t_1, \ldots, \t_d) \in [0, \pi]^d$ 
and $\kb = (k_1,\ldots, k_d)$, then 
\begin{equation}\label{eq:divided-diff}
  \sum_{|\kb| = n, \kb \in \ZZ^d} \e^{\i \kb \cdot \t}  = [\cos \t_1, \cos \t_2, \cdots, \cos \t_d] H_{n,d},
\end{equation}
where $[x_1,\ldots, x_d] f$ denotes the divided difference of $f$ with knots $x_1,\ldots, x_d$ and
\begin{equation*}
    H_{n,d}(\cos \t) =  2 (-1)^{\lfloor \f{d-1}{2}\rfloor} (\sin \t)^{d-1} \times 
      \begin{cases} - \sin (n \t) & \hbox{for $d$ even}, \\ \cos (n \t) & \hbox{for $d$ odd} \end{cases}
\end{equation*}
for $n \ge 1$ and 
$$
   H_{0,d}(\cos \t) =  2 (-1)^{\lfloor \f{d-1}{2}\rfloor} (\sin \t)^{d-1} \cos \f{\t}{2} \times 
      \begin{cases} \cos \f \t 2 & \hbox{for $d$ even}, \\ \sin \f \t 2 & \hbox{for $d$ odd.} \end{cases}
$$
Recall that the divided difference is defined inductively by 
$$
   [x] f = f(x) \quad \hbox{and} \quad [x_0,\ldots,x_m]f = \frac{[x_0,\ldots,x_{m-1}]f -  [x_1,\ldots,x_m]f}{x_0-x_m}. 
$$
It is a symmetric function of the knots $x_0,\ldots, x_m$, which may coalesce. In particular, if all knots coalesce and if 
$f$ is sufficiently differentiable, then the divided difference collapses to
$$
 [x_0,\ldots,x_m]f  = \frac{f^{(m)}(x_0)}{m!} \quad \hbox{if $x_0 = x_1 = \cdots =x_m$}. 
$$
it follows that the cardinality of $\{\kb \in \NN_0^d: |\kb| = n\}$ is given by 
\begin{align} \label{eq:Mnd}
   M_{n,d}^\square \,& = \# \{\kb \in \ZZ^d: |\kb| =n\}= \frac{H_{n,d}^{(d-1)} (1)}{(d-1)!} =: \frac{h_{n,d}}{(d-1)!}. 
\end{align}
for $n \ge 1$ and $M_{0,d}^\square = 1$. The value of $h_{n,d}$ is given by \cite[Lemma 2.3]{BJX}. We shall give another
formula for $h_{n,d}$ after the proof of the following proposition. 

\begin{prop}
For $\x = (\cos \t_1,\ldots, \cos \t_d)$ and $\y = (\cos \phi_1,\ldots, \cos \phi_d)$, then
$$
   \sum_{|\kb| = n, \kb \in \ZZ^d} \e^{\i \kb \cdot (\t - \phi)} = \sum_{\ve \in \{0,1\}^d} 2^{|\ve|}
    \Bigg[ \prod_{j=1}^d  (1-x_j^2)^{\ve_j/2}(1-y_j^2)^{\ve_j/2} \Bigg]
       P_{n-|\ve|}\left(W_{-\mathbf{1}/2+ \ve}^\square; \x,\y\right). 
$$
\end{prop}

\begin{proof}
For $\kb \in \NN_0^d$, we write $\s(\kb) = \{\jb \in \ZZ^d: \jb = (\pm k_1,\ldots, \pm k_d)\}$ be the orbit of $\kb$ under $\ZZ_2^d$. 
For $\kb \in \NN_0^d$ with $k_j \ne 0$, $1 \le j \le d$, 
it is easy to see by symmetry that 
\begin{align*}
  \sum_{ \jb \in \s(\kb) }  \e^{\i \, \jb \cdot (\t - \s)}  \, &= 
  2^d \sum_{k_{i_1},\ldots, k_{i_d}} \prod_{j=1}^\ell \cos( k_{i_j} s_{i_j}) \cos( k_{i_j} t_{i_j})
  \prod_{j=\ell+1}^d \sin( k_{i_j} s_{i_j}) \sin( k_{i_j} t_{i_j}) \\
  & = 
   \sum_{\ve \in \{0,1\}^d} 2^{|\ve|} \prod_{i=1}^d (1-x_i^2)^{\frac{\ve_i}{2}} (1-y_i^2)^{\frac{\ve_i}{2}} P_\kb\left(W_{-\one/2 + \ve}; \x\right) P_\kb \left(W_{-\one/2 + \ve}; \y\right),
\end{align*}
where the sum in the left-hand side of the first equation is over all possible distribution of $\kb = (k_1,\ldots,k_d)$, 
and we have used the fact that $P_\kb\left(W_{-\one/2 + \ve}; \x\right)$ contains $d-|\ve|$ many $\wh T_{k_j}(x_j) = \sqrt{2} \cos(k_j \t_j)$ 
in the second identity. The sum over $\s(\kb)$ with, say, exactly $m$ elements of $\kb$ non-zero is the above identity with  $d$ 
replaced by $m$. Recall that the orthonormal polynomial of degree 0 is equal to 1 by our normalization, we see that the identity 
remains true for $\kb \in \NN_0^d$ that has some zero components. Since summing over $\{\mb \in \ZZ^d: |\mb| =n\}$ is the 
same as summing over $\{\kb \in \NN_0^d: \mb \in \s(\kb), \, |\mb| = n\}$, summing the identity over $\kb \in \NN_0^d$ and $|\kb| =n$
proves the stated identity. 
\end{proof}
 
Setting $\y = \x$ and applying \eqref{eq:divided-diff} in the identity of the above proposition, we obtain the 
extended Pell equations for the cube:
 
\begin{cor} 
For $d \ge 1$, $n = 0,1,2,\ldots$ and $\x\in [-1,1]^d$,  
\begin{align} \label{eq:cube1}
 \sum_{\ve \in \{0,1\}^d} 2^{|\ve|} & \Bigg[ \prod_{j=1}^d  (1-x_j^2)^{\ve}\Bigg]
       P_{n-|\ve|}\left(W_{-\mathbf{1}/2+ \ve}^\square; \x,\x\right)  \\
    &  = \sum_{j=0}^d \binom{d}{j} \left[ \binom{d + n - j}{d} - \binom{d + n -1- j}{d}\right]  \notag
\end{align}
and, summing up the identity,  
\begin{equation} \label{eq:cube2}
 \sum_{\ve \in \{0,1\}^d} 2^{|\ve|}
    \Bigg[ \prod_{j=1}^d  (1-x_j^2)^{\ve}\Bigg]
       K_{n-|\ve|}\left(W_{-\mathbf{1}/2+ \ve}^\square; \x,\x\right) =  \sum_{j=0}^d \binom{d}{j} \binom{d + n - j}{d}.
\end{equation}
\end{cor}
 
\begin{proof}
Setting $\y= \x$ as indicated gives the identity \eqref{eq:cube1} with the right-hand side equal to $M_{n,d}^\square$. Thus,
it is sufficient to establish the identity:
\begin{equation}\label{eq:Mnd2}
 M_{n,d}^\square = \frac{h_{n,d}}{(d-1)!} = \sum_{j=0}^d \binom{d}{j} \left[ \binom{d + n - j}{d} - \binom{d + n -1- j}{d}\right].
\end{equation}
For this we use the following identity established in \cite[Lemma 2.4]{BJX}, 
$$
  \sum_{n=0}^\infty \frac{h_{n,d}(u)}{(d-1)!} r^n = \frac{(1-r^2)^d}{(1-2 r u + r^2)^d}, \qquad 0 \le r < 1,
$$
which implies, in particular, that the following generating function for $h_{n,d}(1)$, 
$$
  \sum_{n=0}^\infty \frac{h_{n,d}(1)}{(d-1)!} r^n = \frac{(1+r)^d}{(1- r)^d}, \qquad 0 \le r < 1,
$$
We now show that the right-hand side of \eqref{eq:Mnd2} satisfies the same generating function. Indeed, a quick 
computation shows that 
\begin{align*}
  \sum_{n=0}^\infty & \sum_{j=0}^d \binom{d}{j} \left[ \binom{d + n - j}{d} - \binom{d + n -1- j}{d}\right]  r^n\\
   &  = (1-r)  \sum_{j=0}^d \binom{d}{j}  \sum_{n=0}^\infty  \binom{d + n - j}{d} r^n \\
   &  = (1-r)  \sum_{j=0}^d \binom{d}{j}  \sum_{n=0}^\infty  \binom{d + n}{d} r^{n+j} \\
   &  = (1-r)  (1+r)^d (1-r)^{-d-1}  = \frac{(1+r)^d}{(1-r)^d}. 
\end{align*}
Since both sides of \eqref{eq:Mnd2} satisfy the same generating function, they are equal. This proves \eqref{eq:cube1}.
Finally, summing over $m$ of \eqref{eq:cube1} proves \eqref{eq:cube2}. 
\end{proof}

As in the two previous cases, \eqref{eq:cube1} is the same as \eqref{th-box-1}, whereas \eqref{eq:cube2}
is equivalent to \eqref{th-box-2} by \eqref{ortho-poly-0}.

\section{Conclusion}

We have considered a set $\Omega \subset\R^d$ that is the unit ball, the simplex, or the unit cube in $\R^d$ in this paper. 
In each case, we proved a remarkable property of the orthonormal polynomials (of all degrees) with respect to the equilibrium 
measure $\mu$ of $\Omega$ and related measures $g\cdot\mu$ on $\Omega$, absolutely continuous with respect to 
 $\mu$, and whose density $g$ is some product of polynomials that define the boundary of $\Omega$. This property is a multivariate analog and generalization of Pell's polynomial equation satisfied by the Chebyshev polynomial 
 of the first and second kind for $\Omega=[-1,1]\subset\R$. The property yields a similar and remarkable property for
 the Christoffel functions associated with those measures, already revealed in \cite{cras} for $\Omega=[-1,1]$ and only 
 partially for the multivariate cases. Moreover, this (identity) property can be interpreted 
 as a distinguished algebraic certificate \emph{\`a la Putinar} that the constant polynomial $\1$ is positive on the set $\Omega$.
 This distinguished representation of $\1$ is a result of a duality between a pair of convex cones, due to Nesterov.
 We hope that our result will stimulate further investigation on the links of the Christoffel functions with those seemingly unrelated 
 fields.

 \subsection*{Acknowledgement}
 This work was initiated when both authors were hosted by the
 American Institute of Mathematics (AIM) in San Jos\'e, CA, during a research week that was part of the SQuaRE program of AIM, and so the authors gratefully acknowledge support from AIM.  The authors also gratefully acknowledge helpful and illuminating comments from N. Levenberg on pluripotential theory  in $\mathbb{C}^d$  and the Riesz $s$-kernel on $\R^d$.

\end{document}